\documentclass[10pt]{amsart}

\usepackage[english]{babel}
\usepackage{amsmath}
\usepackage{amsfonts}
\usepackage{amssymb}
\usepackage{amsthm}
\usepackage{epsfig}
\usepackage{graphicx}
\usepackage{url,hyperref}
\usepackage{color}

\newtheorem{theorem}{Theorem}[section]
\newtheorem{lemma}{Lemma}[section]
\newtheorem{corollary}{Corollary}[section]
\newtheorem{proposition}{Proposition}[section]

\newtheorem{remark}{Remark}[section]
\newtheorem{example}{Example}[section]

\newcounter{theor}

\newtheorem{thm}[theor]{Theorem}

\DeclareMathOperator{\inter}{int}

\def\conv{\mathop\mathrm{conv}\nolimits}
\def\bd{\mathop\mathrm{bd}\nolimits}
\def\supp{\mathop\mathrm{supp}\nolimits}

\def\s{\mathbb{S}}
\def\K{\mathcal{K}}

\def\R{\mathbb{R}}
\def\N{\mathbb{N}}

\def\vol{\mathrm{vol}}

\def\c{\mathcal{C}}

\def\G{\mathrm{G}}

\newcommand{\dlat}{\mathrm{d}}
\DeclareMathOperator*{\esssup}{ess\,sup}

\numberwithin{equation}{section}

\begin{document}
\title{On Rogers-Shephard type inequalities for general measures}

\author[D. Alonso]{David Alonso-Guti\'errez}
\address{Departamento de Matem\'aticas,
Universidad de Zaragoza, 50009-Zaragoza, Spain} \email{alonsod@unizar.es}

\author[M. A. Hern\'andez]{Mar\'\i a A. Hern\'andez Cifre}
\address{Departamento de Matem\'aticas, Universidad de Murcia, Campus de
Espinar\-do, 30100-Murcia, Spain} \email{mhcifre@um.es} \email{jesus.yepes@um.es}

\author[M. Roysdon]{Michael Roysdon}
\address{Department of Mathematical Sciences, Kent State University, Kent,
OH USA} \email{mroysdon@kent.edu} \email{zvavitch@math.kent.edu}

\author[J. Yepes]{Jes\'us Yepes Nicol\'as}

\author[A. Zvavitch]{Artem Zvavitch}

\thanks{First author is supported by MINECO/FEDER project MTM2016-77710-P.
Second and fourth authors are supported by MINECO/FEDER project
MTM2015-65430-P and ``Programa de Ayudas a Grupos de Excelencia de
la Regi\'on de Murcia'', Fundaci\'on S\'eneca, 19901/GERM/15.
Third and fifth authors are supported in part by the U.S. National
Science Foundation Grant DMS-1101636. Fifth author is supported in
part by la Comue Universit\'e Paris-Est.}



\begin{abstract}
In this paper we prove a series of  Rogers-Shephard type inequalities for
convex bodies when dealing with measures on the Euclidean space with
either radially decreasing densities, or quasi-concave densities attaining
their maximum at the origin. Functional versions of classical
Rogers-Shephard  inequalities are also derived as consequences of our
approach.
\end{abstract}

\maketitle

\section{Introduction and main results}

We denote the length of a vector $x \in \R^n$ by $|x|$. We represent by
$B_n=\bigl\{x\in\R^n:|x|\leq 1\bigr\}$ the $n$-dimensional Euclidean unit
ball, by $\s^{n-1}$ its boundary, and $\sigma$ will denote the standard
surface area measure on $\s^{n-1}$. The $n$-dimensional volume of a
measurable set $M\subset\R^n$, i.e., its $n$-dimensional Lebesgue measure,
is denoted by $\vol(M)$ or $\vol_n(M)$ if the distinction of the dimension
is useful (when integrating, as usual, $\dlat x$ will stand for $\dlat
\vol(x)$). With $\inter M$, $\bd M$ and $\conv M$ we denote the interior,
boundary and convex hull of $M$, respectively, and we set $[x,y]$ for
$\conv\{x,y\}$, $x,y\in\R^n$. The set of all $i$-dimensional linear
subspaces of $\R^n$ is denoted by $\G(n,i)$, and for $H\in\G(n,i)$, the
orthogonal projection of $M$ onto $H$ is denoted by $P_HM$. Moreover,
$H^{\bot}\in\G(n,n-i)$ represents the orthogonal complement of $H$.
Finally, let $\K^n$ be the set of all $n$-dimensional convex bodies, i.e.,
compact convex sets with non-empty interior, in $\R^n$. We will frequently
refer to \cite{AGM}, \cite{Ga} and \cite{Sch} for general references for
convex bodies and  their properties.

The Minkowski sum of two non-empty sets $A,B\subset\R^n$ denotes the
classical vector addition of them, $A+B=\{a+b:\, a\in A, \, b\in B\}$, and
we write $A-B$ for $A+(-B)$.

One of the most famous relations involving the volume and the Minkowski addition is
the Brunn-Minkowski inequality (we refer to \cite{G} for an
extensive survey of this inequality). One form of it states that if
$K,L\in\K^n$, then
\begin{equation}\label{e:B-M_ineq}
\vol(K+L)^{1/n}\geq \vol(K)^{1/n}+\vol(L)^{1/n},
\end{equation}
and equality holds if and only if $K$ and $L$ are homothetic.

The Brunn-Minkowski inequality was generalized to different types of
measures, including the case of log-concave measures \cite{Leindler,
Prekopa}, a very powerful generalization to the case of Gaussian measures
\cite{B2, B3, E1,E2, ST}, to $p$-concave measures and many other
extensions (see e.g.~\cite{Borell, BL}). It is interesting to note that it
was proved by Borell \cite{Borell1, Borell} that most of such
generalizations would require a $p$-concavity assumption on the underlined
measure and its density (see \eqref{e:p-concavecondition} below for the
precise definition). Following those works, recently, many classical
results in Convex Geometry were generalized to the case of log-concave
(and in some cases $p$-concave) functions. We mention, among others, the
Blaschke-Santal\'o inequality \cite{AKM,Ba1, FM}, the Bourgain-Milman and
the reverse Brunn-Minkowski inequality \cite{KM}, the general works on
duality and volume \cite{Ba1, Ba2}, as well as the Gr\"unbaum inequality
\cite{MNRY, MSZ} and others \cite{GaZv,LiMaNaZv, Mar, NaTk, RYN}.

In the particular case when $L=-K$, (\ref{e:B-M_ineq}) gives
$$\vol(K-K)\geq 2^n\vol(K),$$ with equality if and only if $K$ is centrally
symmetric, i.e., there exists a point $x \in \R^n$ such that $K-x=-(K-x)$.
An upper bound for the volume of $K-K$ is given by the Rogers-Shephard
inequality, originally proven in \cite[Theorem~1]{RS1}. For more details
about this inequality, we also refer the reader to
\cite[Section~10.1]{Sch} or~\cite{AGM}.

\begin{thm}[The Rogers-Shephard inequality]\label{t:RS}
Let $K\in\K^n$. Then
\begin{equation}\label{e:RS}
\vol(K-K)\leq \binom{2n}{n}\vol(K),
\end{equation}
with equality if and only if $K$ is a simplex.
\end{thm}

Similarly to the Brunn-Minkowski inequality \eqref{e:B-M_ineq},  it is
natural to wonder about the possibility of extending \eqref{e:RS} for
measures associated to certain densities. The most natural candidates
would be the classes of $p$-concave measures. Nevertheless, it was noticed
recently that a number of results in Convex Geometry and Geometric
Tomography can be generalized to a class of measures whose densities have
no concavity assumption. This includes the solution of the Busemann-Petty
problem for general measures \cite{Z}, the Koldobsky slicing inequality
\cite{Kol, KoZ, KK, KLi}, as well as Shephard's problem for general
measures \cite{Liv}.

First we observe that one cannot expect to obtain
\begin{equation}\label{e:RS_not_true}
\mu(K-K)\leq \binom{2n}{n}\mu(K)
\end{equation}
without having certain control on the `position' of the body $K$. Indeed,
it is enough to consider the standard $n$-dimensional Gaussian measure
$\gamma_n$ given~by
\[
\dlat\gamma_n(x)=\frac{1}{(2\pi)^{n/2}}e^{\frac{-|x|^2}{2}}\dlat x,
\]
and $K=x+B_n$ for $|x|$ large enough. In this case it is clear that
$\gamma_n(K-K)=\gamma_n(2B_n)>0$, whereas $\gamma_n(K)$ can be arbitrarily
small.

One option to get control, on the right-hand side of \eqref{e:RS_not_true}
might be to exchange $\mu(K)$ with a mean of the measures of all the
translated copies of $K$ with respect to $-K$. To this end, given a
measure $\mu$ on $\R^n$, we define its \emph{translated-average}
$\overline{\mu}$ as
\begin{equation*}
\overline{\mu}(K)=\dfrac{1}{\vol(K)}\int_{K}\mu(-y+K)\,\dlat y,
\end{equation*}
for any $K\in\K^n$. With this notion, our first main result reads as
follows.
\begin{theorem}\label{t:RS_measures_rad_decreasing}
Let $K\in\K^n$. Let $\mu$ be a measure on $\R^n$ given by
$\dlat\mu(x)=\phi(x)\,\dlat x$, where $\phi:\R^n\longrightarrow[0,\infty)$
is radially decreasing. Then
\begin{equation}\label{e:RS_measures_rad_decreasing}
\mu(K-K)\leq
\binom{2n}{n}\min\bigl\{\overline{\mu}(K),\overline{\mu}(-K)\bigr\}.
\end{equation}
Moreover, if $\phi$ is continuous at the origin then equality holds in
\eqref{e:RS_measures_rad_decreasing} if and only if $\mu$ is a constant
multiple of the Lebesgue measure on $K-K$ and $K$ is a simplex.
\end{theorem}

A function $\phi:\R^n\longrightarrow[0,\infty)$ is said to be radially
decreasing if $\phi(tx)\geq \phi(x)$ for any $t\in[0,1]$ and any point
$x\in\R^n$.

A lower bound for $\mu(K-K)$ when the density function of $\mu$ is even
and $p$-concave (see the definition below), $p\geq -1/n$, can be directly
obtained from the results by Borell and Brascamp-Lieb \cite{Borell,BL}:
\begin{equation}\label{e:B-M(K-K)}
\mu(K-K)\geq\mu(2K).
\end{equation}
Here we extend \eqref{e:B-M(K-K)} to the case of measures with
even and quasi-concave densities (see Theorem
\ref{t:R-S_reverse}).

We recall that a function $\phi:\R^n\longrightarrow[0,\infty)$ is
$p$-concave, for $p\in\R\cup\{\pm\infty\}$, if
\begin{equation}\label{e:p-concavecondition}
\phi\bigl((1-\lambda)x+\lambda y\bigr)\geq
M_p\bigl(\phi(x),\phi(y),\lambda\bigr)
\end{equation}
for all $x,y\in\R^n$ and any $\lambda\in(0,1)$. Here $M_p$ denotes the
{\em $p$-mean} of two non-negative numbers:
\[
M_p(a,b,\lambda)=\left\{
\begin{array}{ll}
\bigl((1-\lambda)a^p+\lambda b^p\bigr)^{1/p}, & \text{ if }p\neq 0,\pm\infty,\\[1mm]
a^{1-\lambda}b^\lambda & \text{ if }p=0,\\[1mm]
\max\{a,b\} & \text{ if }p=\infty,\\[1mm]
\min\{a,b\} & \text{ if }p=-\infty;
\end{array}\right.
\]
for $ab>0$;  $M_p(a,b,\lambda)=0$, when $ab=0$ and
$p\in\R\cup\{\pm\infty\}$. A $0$-concave function is usually called
\emph{log-concave} whereas a $(-\infty)$-concave function is called
\emph{quasi-concave}. Quasi-concavity is equivalent to the fact that the
superlevel sets
\begin{equation}\label{e:ct(phi)}
\c_t(\phi)=\bigl\{x\in\supp\phi:\phi(x)\geq t\|\phi\|_{\infty}\bigr\}
\end{equation}
are convex for $t\in [0,1]$. Here $\supp\phi$ denotes the support of
$\phi$, i.e., the closure of the set $\bigl\{x\in\R^n:\phi(x)>0\bigr\}$,
and with $\|\cdot\|_{\infty}$ we mean
\[
\|\phi\|_{\infty}=\esssup_{x\in\R^n}\phi(x)=\inf\Bigl\{t\in\R:\vol\bigl(\{x\in\R^n:\phi(x)>t\}\bigr)=0\Bigr\}.
\]
We notice that if $\phi$ is $p$-concave, then $\supp\phi$ is a closed
convex set. Furthermore, if a function $\phi$ is quasi-concave and such
that $\max_{x\in\R^n}\phi(x)=\phi(0)$ then it is radially decreasing.

Although the Rogers-Shephard inequality \eqref{e:RS} has been recently
extended to the functional setting (see e.g.~\cite{AAGJV,AlGMJV,Co} and
the references therein), there seems to be no direct way to derive
inequality \eqref{e:RS_measures_rad_decreasing} from the above-mentioned
functional versions just by considering the function $\chi_{_K}\,\phi$,
where $\phi$ is the density of the given measure, and $\chi_{_K}$ is the
characteristic function of a convex body $K$ (see
Remark~\ref{r:functional_NO_RS}). More precisely, in \cite[Theorems~4.3
and~4.5]{Co}, Colesanti extended \eqref{e:RS} to the more general
functional inequality
\begin{equation}\label{e:Colesanti}
\int_{\R^n}\sup_{x=x_1+x_2}\bigl(f(x_1)^p+f(-x_2)^p\bigr)^{1/p}\,\dlat
x\leq  \binom{2n}{n}\int_{\R^n} f(x)\,\dlat x,
\end{equation}
for any $p$-concave integrable function, with $p\in[-\infty,0)$. Here, the
case $p=-\infty$ has to be understood as
$\min\bigl\{f(x_1),f(-x_2)\bigr\}$. In Section \ref{s:radial_decay} we
will also generalize \eqref{e:Colesanti} to general measures (see
Theorem~\ref{t:pquasitheorem}).

In \cite{RS2}, in addition to $K-K$, Rogers and Shephard considered  two
other centrally symmetric convex bodies associated with $K$. The first one
is
\[
CK=\bigl\{(x,\theta)\in\R^{n+1}:\,
x\in(1-\theta)K+\theta(-K),\,\theta\in[0,1]\bigr\},
\]
whose volume is given~by
\[
\vol_{n+1}(CK)=\int_0^1\vol\bigl((1-\theta)K+\theta(-K)\bigr)\,\dlat\theta.
\]
The second one is just $\conv\bigl(K\cup(-K)\bigr)$. The relation of the
volumes of $CK$ and $\conv\bigl(K\cup(-K)\bigr)$ to the volume of $K$ was
proved in \cite{RS2}:
\begin{thm}\label{t:RS_CK}
Let $K\in\K^n$ be a convex body containing the origin. Then
\begin{equation}\label{e:RS_CK}
\int_0^1\vol\bigl((1-\theta)K+\theta(-K)\bigr)\,\dlat\theta
\leq\frac{2^n}{n+1}\,\vol(K),
\end{equation}
with equality if and only if $K$ is a simplex. Moreover,
\begin{equation}\label{e:RS_conv}
\vol\Bigl(\conv\bigl(K\cup(-K)\bigr)\Bigr)\leq2^n\,\vol(K),
\end{equation}
with equality if and only if $K$ is a simplex with the origin as a vertex.
\end{thm}

Here we will show an analog of the above result in the setting of measures
with radially decreasing density:
\begin{theorem}\label{t:RS_CK_conv_hull_rad_dec}
Let $K\in\K^n$ be a convex body containing the origin and let $\mu$ be a
measure on $\R^n$ given by $\dlat\mu(x)=\phi(x)\,\dlat x$, where
$\phi:\R^n\longrightarrow[0,\infty)$ is radially decreasing. Then
\begin{equation}\label{e:R-S_mu_CK}
\int_0^1\mu\bigl((1-\theta)K+\theta(-K)\bigr)\,\dlat\theta
\leq\frac{2^n}{n+1}\,\sup_{\substack{y\in
K\\\theta\in(0,1]}}\frac{\mu\bigl((1-\theta)y-\theta K\bigr)}{\theta^n}
\end{equation}
and
\begin{equation}\label{e:R-S_mu_conv}
\mu\Bigl(\conv\bigl(K\cup(-K)\bigr)\Bigr)\leq 2^n\,\sup_{\substack{y\in
K\\\theta\in(0,1]}}\frac{\mu\bigl((1-\theta)y-\theta K\bigr)}{\theta^n}.
\end{equation}
Moreover, if $\phi$ is continuous at the origin then equality holds in
\eqref{e:R-S_mu_CK} if and only if $\mu$ is a constant multiple of the
Lebesgue measure on $\conv\bigl(K\cup(-K)\bigr)$ and $K$ is a simplex, and
equality holds in \eqref{e:R-S_mu_conv} if and only if $\mu$ is a constant
multiple of the Lebesgue measure on $\conv\bigl(K\cup(-K)\bigr)$ and $K$
is a simplex with the origin as a vertex.
\end{theorem}

We note that the upper bounds in Theorem \ref{t:RS_CK_conv_hull_rad_dec}
are bounded and can be restated using $\|\phi\|_{\infty}\vol(K)$; indeed,
$\mu\bigl((1-\theta)y-\theta K\bigr)/\theta^n$ is bounded from above by
$\|\phi\|_{\infty}\vol(K)$.

In \cite[Theorem~1]{RS2}, Rogers and Shephard also gave the following
lower bound for the volume of $K$ in terms of the volumes of a projection
and a maximal section of $K$:
\begin{thm}\label{t:RS_section_proy}
Let $k\in\{1,\dots,n-1\}$, $H\in\G(n,n-k)$ and $K\in\K^n$. Then
\begin{equation}\label{e:RS_section_proy}
\vol_{n-k}\bigl(P_HK\bigr)\max_{x_0\in H}
\vol_k\bigl(K\cap\bigl(x_0+H^{\bot}\bigr)\bigr)\leq\binom{n}{k}\vol(K).
\end{equation}
\end{thm}

In this paper we will show that the above result remains true for products
of measures associated to quasi-concave densities, provided that
$P_HK\subset K$, i.e., $P_HK=K\cap H$. The assumption on the projection is
necessary, as pointed out in Example \ref{r:hip_P_HK}. In particular, this
hypothesis does not allow one to prove
Theorem~\ref{t:RS_CK_conv_hull_rad_dec} by directly following the proof of
Theorem \ref{t:RS_CK} (see \cite[Theorems~2 and~3]{RS2}): there, the
authors constructed a suitable higher dimensional set to which
\eqref{e:RS_section_proy} was applied. This will be not possible here.

Before stating the result, we fix the following notation: given a convex
body $K$ and $x\in P_HK$, we write $K(x)=(K-x)\cap H^{\bot}$. We will use
the definition of superlevel set $\c_t(\phi)$ given by \eqref{e:ct(phi)}.

\begin{theorem}\label{t:RS_secc_proy_K(0)}
Let $k\in\{1,\dots,n-1\}$ and $H\in\G(n,n-k)$. Given a continuous at the
origin and quasi-concave function $\phi_k:\R^k\longrightarrow[0,\infty)$
with $\|\phi_k\|_{\infty}=\phi_k(0)$ and a radially decreasing function
$\phi_{n-k}:\R^{n-k}\longrightarrow[0,\infty)$, let
$\mu_n=\mu_{n-k}\times\mu_{k}$ be the product measure on $\R^n$ given by
$\dlat\mu_{n-k}(x)=\phi_{n-k}(x)\,\dlat x$ and
$\dlat\mu_{k}(y)=\phi_k(y)\,\dlat y$. Let $K\in\K^n$ with $P_HK\subset K$
and so that
$\vol_k\bigl(\c_t(\phi_k)\cap K(x)\bigr)$ attains its maximum
at $x=0$ for every $t\in(0,1)$. Then
\begin{equation}\label{e:RS_secc_proy_K(0)}
\mu_{n-k}\bigl(P_HK\bigr)\mu_k\bigl(K\cap
H^{\bot}\bigr)\leq\binom{n}{k}\mu_n(K).
\end{equation}
\end{theorem}

The above assumption on the maximal section $K(0)$ of $K$ can be omitted
when the density of the product measure is also quasi-concave, as shown in
Theorem~\ref{t:RS_seccion_proy_quasi}, which is a straightforward
consequence of the following functional version of
\eqref{e:RS_section_proy}.

\begin{theorem}\label{t:functional_RS}
Let $k\in\{1,\dots,n-1\}$ and $H\in\G(n,n-k)$. Let
$f:\R^n\longrightarrow[0,\infty)$ be a bounded quasi-concave function such
that $\vol_k\bigl(\c_t(f)\cap(x+H^{\bot})\bigr)$, $x\in H$, attains its
maximum at $x=0$ for every $t\in(0,1)$, and let
$g:H\longrightarrow[0,\infty)$ be a radially decreasing function. Then,
\begin{equation*}\label{e:proy_sect_f_g}
\int_{H}g(x)P_Hf(x)\,\dlat x\int_{H^{\bot}}f(y)\,\dlat y
\leq\binom{n}{k}\|f\|_{\infty}\int_{\R^n}g(P_Hx)f(x)\,\dlat x.
\end{equation*}
\end{theorem}
Here, the projection function $P_Hf:H\longrightarrow[0,\infty)$ of $f$ is
defined by $P_Hf(x)=\sup_{y\in H^{\bot}}f(x+y)$.

In the particular case of a log-concave integrable function $f$, this
result has been recently obtained in \cite[Theorem~1.1]{AAGJV}.

The paper is organized as follows. Section \ref{s:radial_decay} is mainly
devoted to the proofs of Theorems \ref{t:RS_measures_rad_decreasing} and
\ref{t:RS_CK_conv_hull_rad_dec} as well as the functional analogs of these
results. We start Section~\ref{s:functions} by deriving a general result
for functions with certain concavity conditions, which will play a
relevant role along the manuscript. As a consequence of this result we
prove, in particular, Theorem \ref{t:functional_RS}. Next, in Section
\ref{s:quasi_concave}, we study  Rogers-Shephard type inequalities for
measures  with quasi-concave densities, and prove  Theorem
\ref{t:RS_secc_proy_K(0)}. Finally, in Section \ref{s:remark}, we present
another Rogers-Shephard type inequality when assuming a further concavity
for the density of the involved measure.

\section{Rogers-Shephard type inequalities for measures with radially decreasing densities}\label{s:radial_decay}

\subsection{The case of convex sets}\label{ss:R-S_sets}

As pointed out in the previous section, one cannot expect to obtain
\eqref{e:RS_not_true} without having control on the translations of the
set $K$. Moreover, certain requirements on the density of the measure
$\mu$ must be made (see also the comments after
Remark~\ref{c:RS_measures_rad_decay} and Example~\ref{ex:ring}). To this
regard, in Section \ref{s:quasi_concave} we will show that one may
consider quasi-concave densities with maximum at the origin. In this
setting, we will also obtain other Rogers-Shephard type inequalities.

Let us now  follow a different approach. First we will prove an extension
of \eqref{e:RS} for the more general case of radially decreasing
densities, collected in Theorem \ref{t:RS_measures_rad_decreasing}. Before
showing it, we need the following auxiliary result.

\begin{lemma}\label{l:F_nonposit}
Let $\phi:[0,\infty)\longrightarrow[0,\infty)$ be a decreasing function
and let $n,m\in\N$. Then, for every $x\in(0,\infty)$,
\begin{equation*}
\int_0^x \left(1-\frac{t}{x} \right)^n t^{m-1} \phi(t)\,\dlat t \geq
\binom{n+m}{n}^{-1} \int_0^x t^{m-1}\phi(t) \,\dlat t,
\end{equation*}
with equality if and only if $\phi$ is constant on $(0,x)$.
\end{lemma}
\begin{proof}
Considering the function $F:(0,\infty)\longrightarrow[0,\infty)$ given by
\begin{equation*}
F(x)=\binom{n+m}{n}^{-1}\int_0^x t^{m-1}\phi(t) \,\dlat t -\int_0^x
\left(1-\frac{t}{x} \right)^n t^{m-1} \phi(t)\,\dlat t,
\end{equation*}
we need to show that it is non-positive.

Expanding the binomial $\left(1 - t/x \right)^n$ we may assert on one hand
that $F(x)\to 0$ as $x\to 0^+$. On the other hand, and jointly with
Lebesgue's differentiation theorem, we get that the derivative of $F$
exists for almost every $x\in(0,\infty)$ and further
\[
F'(x)=\binom{n+m}{n}^{-1}\,x^{m-1}\phi(x)-n\int_0^x\left(1-\frac{t}{x}\right)^{n-1}\frac{t^m}{x^2}\,\phi(t)
\,\dlat t.
\]
Now, applying the change of variable $u = t/x$, we get
\[
n\int_0^x\left(1-\frac{t}{x}\right)^{n-1}t^m\,\dlat
t=\dfrac{n\,\Gamma(n)\Gamma(m+1)}{\Gamma(n+m+1)}x^{m+1}=\binom{n+m}{n}^{-1}x^{m+1},
\]
where $\Gamma$ represents the Gamma function. This together with the fact
that $\phi$ is decreasing implies that $F'(x)\leq0$, with equality if and
only if $\phi$ is constant on $(0,x)$.

Since $F$ is absolutely continuous on every interval
$[a,b]\subset(0,\infty)$, because it arises as a finite sum of products of
absolutely continuous functions,
\[
F(x)=F(a)+\int_a^xF'(s)\,\dlat s\leq F(a)
\]
for all $x>0$ and any $0<a\leq x$. Taking into account that
$\lim_{a\to0^+}F(a)=0$ we then have
\[
F(x)= \int_0^x F'(s)\,\dlat s\leq 0,
\]
with equality if and only if $F'\equiv0$ almost everywhere or,
equivalently, when $\phi$ is constant on $(0,x)$.
\end{proof}

Next we prove Theorem \ref{t:RS_measures_rad_decreasing}. We follow the
idea of the original proof of the Rogers-Shephard inequality (\cite{RS1}),
with the main difference of the application of Lemma \ref{l:F_nonposit} in
\eqref{e:proving_RS_polar_lemma}.

\begin{proof}[Proof of Theorem \ref{t:RS_measures_rad_decreasing}]
Let $f:\R^n\longrightarrow[0,\infty)$ be the function given by
\[
f(x)=\vol\bigl(K\cap(x+K)\bigr).
\]
Observe that $\supp f=K-K$ and $f$ vanishes on ${\rm bd}(K-K)$.
Furthermore, using the  Brunn-Minkowski inequality~\eqref{e:B-M_ineq}
together with the inclusion
\begin{equation}\label{eq:inclus}
K\cap\bigl[(1-\lambda)x+\lambda y+K\bigr]
\supset(1-\lambda)\bigl[K\cap(x+K)\bigr]+\lambda\bigl[K\cap(y+K)\bigr],
\end{equation}
which holds for all $\lambda\in[0,1]$ and $x,y\in K-K$, we get that $f$ is
$(1/n)$-concave.

On the one hand, by Fubini's theorem, we have
\begin{equation}\label{e:proving_RS_fubini}
\begin{split}
\int_{K-K}f(x) \,\dlat\mu(x) & = \int_{\R^n}\int_{\R^n}\chi_{_K}(y)\chi_{_{y-K}}(x)\,\phi(x)\,\dlat
y\,\dlat x\\
 & =\int_{K}\mu(y-K)\,\dlat y=\vol(K)\,\overline{\mu}(-K).
\end{split}
\end{equation}
On the other hand, we define the function $g:K-K\longrightarrow[0,\infty)$
given by
\[
g(x)=f(0)\left[1-\frac{|x|}{\rho_{_{\!K-K}}\bigl(x/|x|\bigr)} \right]^n,
\quad \text{for every } x\neq0,
\]
and $g(0)=f(0)$, where
\[
\rho_{_{\!L}}(u)=\max\{\rho\geq 0:\rho u\in L\},\quad u\in\s^{n-1},
\]
stands for the radial function of $L\in\K^n$. Notice that $g^{1/n}$ is
affine on $\bigl[0,\rho_{_{\!K-K}}(u)u\bigr]$, for all $u\in\s^{n-1}$, and
so $g(0)^{1/n}=f(0)^{1/n}$ and
\[
g\bigl(\rho_{_{\!K-K}}(u)u\bigr)^{1/n}=0=f\bigl(\rho_{_{\!K-K}}(u)u\bigr)^{1/n}.
\]
Hence, since $f^{1/n}$ is concave, it follows that $f^{1/n}\geq g^{1/n}$
on $\bigl[0,\rho_{_{\!K-K}}(u)u\bigr]$. Therefore, using polar
coordinates, we have
\begin{equation}\label{e:proving_RS_polar}
\begin{split}
\int_{K-K}f(x)\,\dlat\mu(x) &
    =\int_{\s^{n-1}}\int_0^{\rho_{_{\!K-K}}(u)} r^{n-1} f(r u) \phi(r u)\,\dlat r\,\dlat\sigma(u)\\
 & \geq f(0)\int_{\s^{n-1}}\int_0^{\rho_{_{\!K-K}}(u)} \left(1 - \frac{r}{\rho_{_{\!K-K}}(u)}\right)^n r^{n-1} \phi(r u) \,\dlat r \,\dlat\sigma(u).
\end{split}
\end{equation}
Now, from \eqref{e:proving_RS_polar} and Lemma \ref{l:F_nonposit} we
obtain
\begin{equation}\label{e:proving_RS_polar_lemma}
\begin{split}
\int_{K-K} f(x) \,\dlat\mu(x)
& \geq \frac{1}{\binom{2n}{n}}f(0)\int_{\s^{n-1}}\int_0^{\rho_{_{\!K-K}}(u)}r^{n-1}
\phi(r u)\,\dlat r\,\dlat\sigma(u)\\
& =\frac{1}{\binom{2n}{n}}\vol(K)\mu(K-K),
\end{split}
\end{equation}
which, together with \eqref{e:proving_RS_fubini}, yields
\[
\mu(K-K)\leq\binom{2n}{n}\overline{\mu}(-K).
\]
By replacing $K$ with $-K$, we obtain the desired inequality.

Finally we notice that equality holds in
\eqref{e:RS_measures_rad_decreasing} only if there is equality in
\eqref{e:proving_RS_polar_lemma}. This implies, by Lemma
\ref{l:F_nonposit}, that $\phi(ru)$ is constant on
$\bigl(0,\rho_{_{\!K-K}}(u)\bigr)$ for $\sigma$-almost every
$u\in\s^{n-1}$. Since $\phi$ is continuous at the origin, $\mu$ is a
constant multiple of the Lebesgue measure on $K-K$ and, by Theorem
\ref{t:RS}, $K$ is a simplex. The converse immediately follows from
Theorem \ref{t:RS}.
\end{proof}

\begin{remark}
From the proof of the equality case in the above result (and the
corresponding one of Lemma~\ref{l:F_nonposit}), we notice that the
assumption of continuity at the origin for $\phi$ is necessary in order to
`recover' the Lebesgue measure (up to a constant). Indeed, one could
consider a simplex $K$ and a function $\phi$ that is constant on
$\bigl(0,\rho_{_{\!K-K}}(u)\bigr)$ for every $u\in\s^{n-1}$, but not
necessarily constant on $K-K$, and thus
\eqref{e:RS_measures_rad_decreasing} would hold with equality.
\end{remark}

The next theorem is obtained just by repeating the same argument given in
the proof of Theorem \ref{t:RS_measures_rad_decreasing}, but replacing
$-K$ with $L$.
\begin{theorem}\label{c:RS_K_L} Let $K, L\in\K^n$ and let
$\mu$ be a measure on $\R^n$ given by $\dlat\mu(x)=\phi(x)\,\dlat x$,
where $\phi:\R^n\longrightarrow[0,\infty)$ is radially decreasing. Then
\[
\mu(K+L)\vol\bigl(K\cap(-L)\bigr)\leq \binom{2n}{n}\int_K\mu(x+L)\dlat x.
\]
\end{theorem}

\begin{remark}\label{c:RS_measures_rad_decay}
As a straightforward consequence of Theorem
\ref{t:RS_measures_rad_decreasing}, we get the following statement. Let
$K\in\K^n$ and let $\mu$ be a measure on $\R^n$ given by
$\dlat\mu(x)=\phi(x)\,\dlat x$, where $\phi:\R^n\longrightarrow[0,\infty)$
is radially decreasing. Then
\begin{equation}\label{e:RS_measures_rad_decay_coro}
\mu(K-K)\leq \binom{2n}{n}
\min\left\{\sup_{x\in\R^n}\mu(x+K),\sup_{x\in\R^n}\mu(x-K)\right\}.
\end{equation}
\end{remark}

The above fact trivially holds in dimension $n=1$ for an arbitrary
measure. Indeed, given $K=[a,b]$, then
\[
\begin{split}
\mu(K-K)=\mu\bigl([a-b,b-a]\bigr)&=\mu\bigl([a,b]-a\bigr)+\mu\bigl([a,b]-b\bigr)\\
&\leq2\min\left\{\sup_{x\in\R}\mu(x+K),\,\sup_{x\in\R}\mu(x-K)\right\}.
\end{split}
\]
However, in dimension $n\geq 2$ the radial decay assumption cannot be
omitted, as the following example shows.

\begin{example}\label{ex:ring}
Fix
$0<\varepsilon<\delta<2$. Consider the measure $\mu$  on $\R^2$ with density
\[
\phi(x)=\left\{\begin{array}{ll}
1 & \text{ if }x\in\delta B_2\cup\bigl(2B_2\setminus(2-\varepsilon)B_2\bigr),\\
0 & \text{ otherwise}
\end{array}\right.
\]
(see Figure
\ref{f:example_dim_2_assump_density_needed}). Then
\begin{equation}\label{e:ex_packing}
\mu(B_2-B_2)>6\sup_{x\in\R^2}\mu(x+B_2).
\end{equation}
Note that \eqref{e:ex_packing} contradicts
\eqref{e:RS_measures_rad_decay_coro}. Indeed, on the one hand,
\[
\mu(B_2-B_2)=\mu(2B_2)=\pi\delta^2+\bigl(4-(2-\varepsilon)^2\bigr)\pi=4\pi\varepsilon+\pi(\delta^2-\varepsilon^2).
\]

\begin{figure}[h]
\begin{center}
\includegraphics[width=3.8cm]{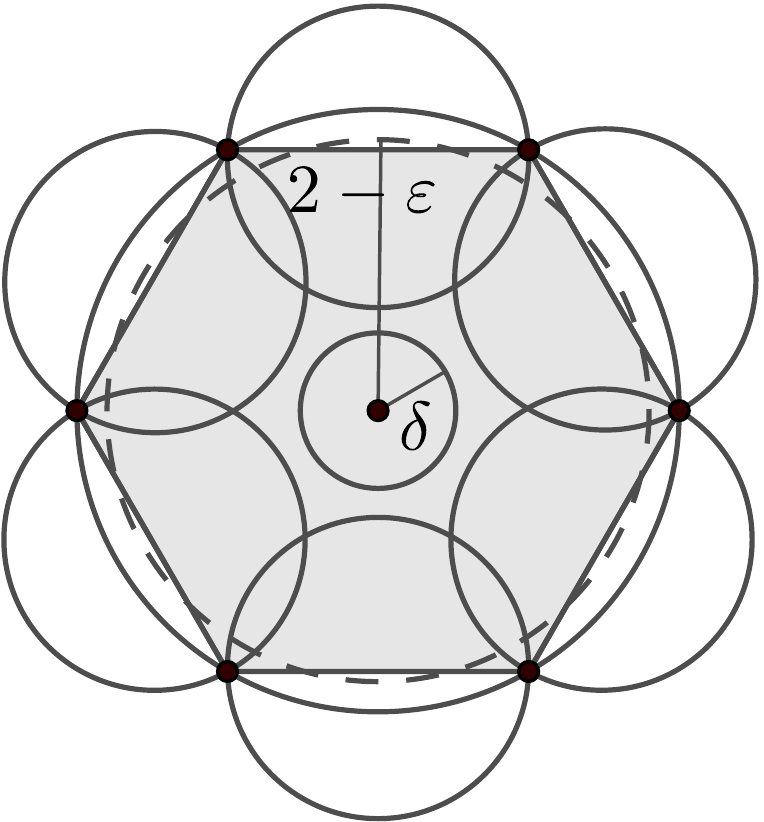}
\end{center}
\caption{Constructing a measure for which
\eqref{e:RS_measures_rad_decay_coro} does not hold.}
\label{f:example_dim_2_assump_density_needed}
\end{figure}

On the other hand, we note that we need at least 6 copies of the unit disk
in order to cover $\bd(2B_2)$, which can be seen by considering a regular
hexagon inscribed in $2B_2$ (see Figure
\ref{f:example_dim_2_assump_density_needed}). Moreover,  if we would cover
$\bd(2B_2)$ with exactly $6$ translated copies of $B_2$, then the covering
discs would stay away from the origin. Thus, for $\varepsilon>0$ small
enough,
\[
\sup_{x\in\R^2}\vol\Bigl((x+B_2)\cap\bigl(2B_2\setminus(2-\varepsilon)B_2\bigr)\Bigr)=\frac{1}{6}4\pi\varepsilon+o(\varepsilon).
\]
Taking, e.g., $\delta =\sqrt{\varepsilon}/100$ we get, for $\varepsilon$
small enough, that $\delta >\varepsilon$, and also that
$4\pi\varepsilon/6>\pi\delta^2$ and $o(\varepsilon)<\delta^2$.  Thus
\[
\begin{split}
6\sup_{x\in\R^2}\mu(x+B_2)&=6\sup_{x\in\R^2}\vol\Bigl((x+B_2)\cap\bigl(2B_2\setminus(2-\varepsilon)B_2\bigr)\Bigr)
=4\pi\varepsilon+o(\varepsilon)\\
&<4\pi\varepsilon+\pi(\delta^2-\varepsilon^2).
\end{split}
\]
Moreover, since $\sup_{x\in\R^2}\mu(x+B_2)>\overline{\mu}(B_2)$, this example
shows that the radial decay assumption is also needed in
Theorem~\ref{t:RS_measures_rad_decreasing}.
\end{example}

Regarding a reverse inequality for Theorem
\ref{t:RS_measures_rad_decreasing} (or
\eqref{e:RS_measures_rad_decay_coro}), we have the following result, which
extends \eqref{e:B-M(K-K)}.

\begin{theorem}\label{t:R-S_reverse}
Let $K\in\K^n$. Let $\mu$ be a measure on $\R^n$ given by
$\dlat\mu(x)=\phi(x)\,\dlat x$, where $\phi:\R^n\longrightarrow[0,\infty)$
is an even quasi-concave function. Then
\begin{equation}\label{e:reverse_RS_quasi}
\mu(K-K)\geq\mu(2K).
\end{equation}
Equality holds in \eqref{e:reverse_RS_quasi} only if $K\cap(\supp\phi)/2$
is centrally symmetric. Moreover, if $K$ is centrally symmetric with
respect to the origin, then equality holds in \eqref{e:reverse_RS_quasi}.
\end{theorem}

\begin{proof}
We write $\overline{K}_t=(2K)\cap\c_t(\phi)$ for every $t\in[0,1]$. On the
one hand, by Fubini's theorem, we have
\begin{equation}\label{e:proving_reverse_RS_quasi}
\begin{split}
\mu(2K) & = \int_{2K} \phi(x) \,\dlat x
    =\|\phi\|_{\infty}\int_{2K}\int_0^{\frac{\phi(x)}{\|\phi\|_{\infty}}}\,\dlat t \,\dlat x
    =\|\phi\|_{\infty} \int_0^1 \int_{2K} \chi_{_{\c_t(\phi)}}(x) \,\dlat x \,\dlat t\\
   & =\|\phi\|_{\infty} \int_0^1 \vol\bigl(\overline{K}_t\bigr) \,\dlat t
   \leq\|\phi\|_{\infty} \,2^{-n} \int_0^1 \vol\bigl(\overline{K}_t-\overline{K}_t\bigr) \,\dlat t,
\end{split}
\end{equation}
where in the last inequality we have used the Brunn-Minkowski inequality
(cf.~\eqref{e:B-M_ineq}).

On the other hand, since $\phi$ is quasi-concave and even, then
$\c_t(\phi)$ is convex and centrally symmetric (with respect to the
origin), and hence $\overline{K}_t-\overline{K}_t\subset(2K-2K)\cap
2\c_t(\phi)=2\bigl((K-K)\cap\c_t(\phi)\bigr)$. Thus, we get
\begin{equation*}
\begin{split}
\mu(2K) &\leq \|\phi\|_{\infty} \,2^{-n} \int_0^1
\vol\bigl(\overline{K}_t-\overline{K}_t\bigr) \,\dlat t
\leq \|\phi\|_{\infty} \int_0^1 \vol \bigl( (K-K) \cap \c_t(\phi)\bigr) \,\dlat t\\
&=  \|\phi\|_{\infty} \int_0^1 \int_{\R^n} \chi_{_{(K-K) \cap
\c_t(\phi)}}(x) \,\dlat x \,\dlat t = \mu(K-K).
\end{split}
\end{equation*}
For the equality case, we note that the identity $\mu(2K)=\mu(K-K)$
implies that \eqref{e:proving_reverse_RS_quasi} holds with equality, and
thus
$\vol\bigl(\overline{K}_t\bigr)=2^{-n}\vol\bigl(\overline{K}_t-\overline{K}_t\bigr)$
for almost every $t\in[0,1]$. Then, there exists a decreasing sequence
$(t_m)_m\subset[0,1]$ with $t_m\to 0$ and such that
$\vol\bigl(\overline{K}_{t_m}\bigr)=2^{-n}\vol\bigl(\overline{K}_{t_m}-\overline{K}_{t_m}\bigr)$
for all $m\in\N$. Therefore, since the boundary of a convex set has null
(Lebesgue) measure, we get
\begin{equation}\label{e:lower_bound}
\begin{split}
\vol\bigl((2K)\cap\supp\phi\bigr) &
    =\vol\left(\bigcup_{m=1}^\infty\overline{K}_{t_m}\right)
    =\lim_m \vol\bigl(\overline{K}_{t_m}\bigr)
    =\lim_m 2^{-n}\vol\bigl(\overline{K}_{t_m}-\overline{K}_{t_m}\bigr)\\
 & =2^{-n}\vol\left(\bigcup_{m=1}^\infty\bigl(\overline{K}_{t_m}-\overline{K}_{t_m}\bigr)\right)\\
  & =2^{-n}\vol\Bigl(\bigl((2K)\cap\supp\phi\bigr)-\bigl((2K)\cap\supp\phi\bigr)\Bigr).
\end{split}
\end{equation}
Since $\supp\phi$ is an $n$-dimensional convex set containing the origin
then $\mu(2K)=\mu(K-K)>0$, and so $\vol\bigl((2K)\cap\supp\phi\bigr)>0$.
Therefore \eqref{e:lower_bound} implies that $(2K)\cap\supp\phi$ is
centrally symmetric. The sufficient condition is evident.
\end{proof}

If we apply \eqref{e:reverse_RS_quasi} to the set $K'=K+x/2$ then
$\mu(K-K)\geq\sup_{x\in\R^n}\mu(x+2K)$ also holds. We observe, however,
that we cannot expect a general reverse inequality for
\eqref{e:RS_measures_rad_decay_coro} in the non-even case, as the
following example shows.

\begin{example}
Let $\theta>0$ and consider $W_{\theta}=\bigl\{r(\cos t,\sin t):0\leq
t\leq\theta,\,r\geq 0\bigr\}\subset\R^2$. Let $\mu_{\theta}$ be the
measure on $\R^2$ with density $\phi_{\theta}(x)=\chi_{_{W_\theta}}(x)$
(see Figure \ref{f:wedge}).

\begin{figure}[h]
\begin{center}
\includegraphics[width=7.8cm]{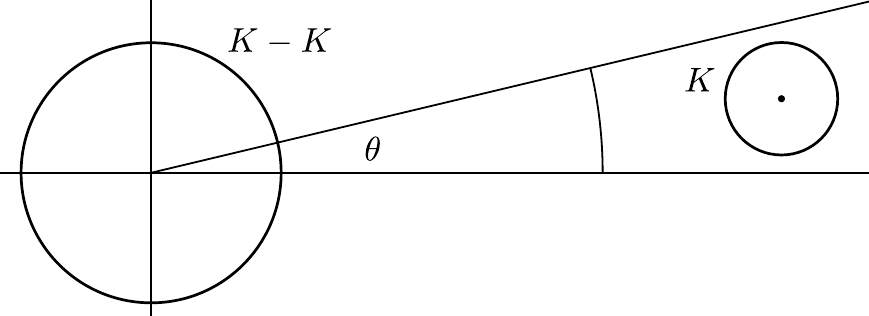}
\end{center}
\caption{A construction for which $\mu(K-K)\to 0$.} \label{f:wedge}
\end{figure}

By letting $\theta\to 0$, we can move a set $K$ far enough, but keeping
the measure of the shifts of $K$ constant, while the measure of $K-K$ will
be arbitrarily small. So the left-hand side of
\eqref{e:RS_measures_rad_decay_coro} tends to zero whereas the right-hand
side is fixed.
\end{example}

A way to strengthen inequality \eqref{e:RS_measures_rad_decay_coro} would
be to replace $\mu(K-K)$ by the quantity $\sup_{\omega\in\R^n} \mu(K-K +\omega)$:

\medskip

\noindent{\bf Question:} Given a measure $\mu$ on $\R^n$, is it true that for every $K\in\K^n$
\begin{equation*}\label{e:question}
\sup_{\omega\in\R^n} \mu(K-K
+\omega)\leq\binom{2n}{n}\min\left\{\sup_{x\in\R^n}\mu(x+K),\sup_{x\in\R^n}\mu(x-K)\right\}?
\end{equation*}

The following result partially solves this question, in the setting of
quasi-concave densities, by exploiting the approach carried out in the
proof of Theorem \ref{t:RS_measures_rad_decreasing}. The idea relies on
the possibility of finding a point, for each translated copy of $K-K$,
from which the density is radially decreasing over the given translation
of $K-K$. The negative counterpart is the apparent necessity of including
a factor jointly with the measure of the shift of $K-K$. Nevertheless, we
observe that the supremum on the right-hand side can be taken over $K$. In
Section \ref{s:quasi_concave}, we will provide a different solution to
this issue (see Theorem~\ref{t:RS_omega_quasiconcave}).

\begin{proposition}\label{t:RS_omega_rad_decreasing}
Let $K\in\K^n$ and let $\mu$ be a measure on $\R^n$ given by
$\dlat\mu(x)=\phi(x)\,\dlat x$, where $\phi:\R^n\longrightarrow[0,\infty)$
is a quasi-concave function whose restriction to its support is
continuous. Then, for every $\omega\in\R^n$,
\begin{equation}\label{e:RS_omega_rad_decreasing}
c(\omega)\mu(K-K+\omega) \leq\binom{2n}{n}\sup_{y\in K}\mu(y+\omega-K),
\end{equation}
where $c(\omega)=\vol\bigl(K\cap(\omega'-\omega+K)\bigr)\vol(K)^{-1}$, and
$\omega'\in K-K+\omega$ is such that $\phi(\omega')=\max_{x\in
K-K+\omega}\phi(x)$. Moreover, equality holds for some $\omega_0\in\R^n$
if and only if $\mu$ is a constant multiple of the Lebesgue measure on
$K-K+\omega_0$, $c(\omega_0)=1$ and $K$ is a simplex.
\end{proposition}

\begin{proof}
Let $f:\R^n\longrightarrow[0,\infty)$ be defined as
$f(x)=\vol\bigl(K\cap(x-\omega+K)\bigr)$. As before, we get that $\supp
f=K-K+\omega$ and  $f$ is $(1/n)$-concave (see  \eqref{e:B-M_ineq} and
(\ref{eq:inclus})). On the one hand, by Fubini's theorem, we have
\begin{equation}\label{e:proving_RS_fubini_2}
\int_{K-K+\omega}f(x) \,\dlat\mu(x)=
\int_{\R^n}\int_{\R^n}\chi_{_K}(y)\chi_{_{y+\omega-K}}(x)\,\phi(x)\,\dlat
y\,\dlat x=\int_{K}\mu(y+\omega-K)\,\dlat y.
\end{equation}
On the other hand, from the continuity of $\phi$  on $\supp\phi$, we know
that there exists a point $\omega'\in(K-K+\omega)\cap\supp\phi$, which is
a compact set, such that $\phi(\omega') = \max_{x\in K-K+\omega}\phi(x)$.
This, together with the quasi-concavity of $\phi$,  implies that it
radially decays from $\omega'$ on $K-K+\omega$, i.e.,
$\phi\bigl(\omega'+t(x-\omega')\bigr)\geq\phi(x)$ for any $t\in[0,1]$ and
all $x\in K-K+\omega$.

Now we define the function $g:K-K+\omega\longrightarrow[0,\infty)$ given
by
\[
g(x)=f(\omega')\left[1-\frac{|x-\omega'|}{\rho_{_{\!K-K+\omega-\omega'}}\bigl((x-\omega')/|x-\omega'|\bigr)}
\right]^n, \quad \text{ for every } x\neq\omega',
\]
and $g(\omega')=f(\omega')$. Since $f^{1/n}$ is concave, it follows that
$f^{1/n}\geq g^{1/n}$ on
$\bigl[\omega',\omega'+\rho_{_{\!K-K+\omega-\omega'}}(u)u\bigr]$, and so,
via the polar coordinates $z=x-\omega'=ru$, we get
\begin{equation*}\label{e:proving_RS_polar_2}
\begin{split}
& \int_{K-K+\omega} f(x) \,\dlat\mu(x) =\int_{K-K+\omega-\omega'} f(\omega'+z)\phi(\omega'+z)\,\dlat z\\
 & =\int_{\s^{n-1}}\int_0^{\rho_{_{\!K-K+\omega-\omega'}}(u)}r^{n-1}f(\omega'+ru)\phi(\omega'+ru)\,\dlat r\,\dlat\sigma(u)\\
 & \geq f(\omega')\int_{\s^{n-1}}\int_0^{\rho_{_{\!K-K+\omega-\omega'}}(u)}
    \left[1-\frac{r}{\rho_{_{\!K-K+\omega-\omega'}}(u)}\right]^nr^{n-1}\phi(\omega'+ru)\,\dlat r \,\dlat\sigma(u).
\end{split}
\end{equation*}
Then Lemma \ref{l:F_nonposit} yields
\begin{equation}\label{e:proving_RS_omega_lemma}
\begin{split}
\int_{K-K+\omega} f(x) \,\dlat\mu(x)
 & \geq\frac{f(\omega')}{\binom{2n}{n}}\int_{\s^{n-1}}\int_0^{\rho_{_{\!K-K+\omega-\omega'}}(u)}
    r^{n-1} \phi(\omega'+r u)\,\dlat r\,\dlat\sigma(u)\\
 & =\frac{1}{\binom{2n}{n}}\vol\bigl(K\cap(\omega'-\omega+K)\bigr)\mu(K-K+\omega),
\end{split}
\end{equation}
which, together with \eqref{e:proving_RS_fubini_2}, gives
\[
\begin{split}
\mu(K-K+\omega)\vol\bigl(K\cap(\omega'-\omega+K)\bigr)
& \leq\binom{2n}{n}\int_{K}\mu(y+\omega-K)\,\dlat y\\
& \leq\binom{2n}{n}\vol(K)\sup_{y\in K}\mu(y+\omega-K).
\end{split}
\]
Finally we notice that equality holds in \eqref{e:RS_omega_rad_decreasing}
for some $\omega_0\in\R^n$ only if there is equality in
\eqref{e:proving_RS_omega_lemma}. This implies, by Lemma
\ref{l:F_nonposit}, that $\phi(\omega'+r u)$ is constant on
$\bigl(0,\rho_{_{\!K-K+\omega_0-\omega'}}(u)\bigr)$ for $\sigma$-almost
every $u\in\s^{n-1}$. Since $\phi$ is continuous at $\omega'\in\supp\phi$,
$\mu$ is a constant multiple of the Lebesgue measure on $K-K+\omega_0$
and, by Theorem \ref{t:RS}, $K$ is a simplex (in particular,
$c(\omega_0)=1$). The converse immediately follows from Theorem
\ref{t:RS}.
\end{proof}

\subsection{The functional case}

In this subsection we draw a consequence of Theorem
\ref{t:RS_measures_rad_decreasing} regarding integrals of quasi-concave
functions, which extends two results of Colesanti \cite[Theorems~4.3 and
4.5]{Co} and is collected in Theorem~\ref{t:pquasitheorem}. To this end,
given a quasi-concave function $f:\R^n\longrightarrow[0,\infty)$, we
define the ($-\infty$)-difference of $f$, which remains quasi-concave (cf.
\cite[Proposition~4.2]{Co}), by
\[
\Delta_{-\infty}f(z)=\sup_{z=x-y}\min\bigl\{f(x),f(y)\bigr\}.
\]
Besides $\Delta_{-\infty} f$, we
also consider the (difference) functions $\Delta_{-\infty,\theta} f$ (for
some $\theta\in [0,1]$) and $\widetilde{\Delta}_{-\infty} f$ given by
\begin{equation*}
\begin{split}
\Delta_{-\infty,\theta} f(z) & =\sup_{z=(1-\theta)x-\theta y}\min\bigl\{f(x),f(y)\bigr\},\\
\widetilde{\Delta}_{-\infty} f(z) &
    =\sup_{\substack{z=(1-\theta)x-\theta y\\\theta\in[0,1]}}\min\bigl\{f(x),f(y)\bigr\}.
\end{split}
\end{equation*}
These functions can be regarded as the (quasi-concave) functional
counterparts of $K-K$, $(1-\theta)K-\theta K$ and
$\conv\bigl(K\cup(-K)\bigr)$, respectively, as it is shown via their
(strict) superlevel sets. For the sake of brevity we will write, for a
function $f:\R^n\longrightarrow[0,\infty)$ and $t\in[0,\infty)$,
\[
S_{>t}(f)=\bigl\{x\in\R^n:f(x)>t\bigr\};
\]
analogously, $S_{\geq t}(f)=\bigl\{x\in\R^n:f(x)\geq t\bigr\}$. We observe that if
$f:\R^n\longrightarrow[0,\infty)$ is a quasi-concave function, then
\begin{equation}\label{e:superlevel_quasi}
\begin{split}
(i) & \qquad S_{>t}\bigl(\Delta_{-\infty}f\bigr)=S_{>t}(f)-S_{>t}(f),\\
(ii) & \qquad S_{>t}\bigl(\Delta_{-\infty,\theta}f\bigr)=(1-\theta)S_{>t}(f)-\theta S_{>t}(f),\\
(iii) & \qquad S_{>t}\left(\widetilde{\Delta}_{-\infty}f\right)=\conv\Bigl(S_{>t}(f)\cup\bigl(-S_{>t}(f)\bigr)\Bigr).
\end{split}
\end{equation}
Indeed, (i), (ii) and (iii) are completely analogous. To see (i), let
$z\in S_{>t}\bigl(\Delta_{-\infty}f\bigr)$. Then there exist $x,y$ such
that $z=x-y$ and $\min\bigl\{f(x),f(y)\bigr\}>t$, which shows the
inclusion
\[
S_{>t}\bigl(\Delta_{-\infty}f\bigr)\subset S_{>t}(f)-S_{>t}(f).
\]
For the reverse inclusion, if $z\in S_{>t}(f)-S_{>t}(f)$ then there exist
$x,y\in\R^n$, with $z=x-y$, such that $f(x)>t$ and $f(y)>t$. Since
$\min\bigl\{f(x),f(y)\bigr\}>t$ and $z=x-y$, we get that
$\Delta_{-\infty}f(z)>t$, as desired.

\medskip

Now we collect the above-mentioned consequence of
\eqref{e:RS_measures_rad_decreasing}, which may be seen as its functional
version.

\begin{theorem}\label{t:pquasitheorem}
Let $f:\R^n\longrightarrow[0,\infty)$ be an integrable quasi-concave
function. Let $\mu$ be a measure on $\R^n$
given by $\dlat\mu(x)=\phi(x)\,\dlat x$, where
$\phi:\R^n\longrightarrow[0,\infty)$ is radially decreasing. Then
\begin{equation}\label{e:pquasitheorem}
\int_{\R^n}\Delta_{-\infty}f(x)\,\dlat\mu(x)\leq\binom{2n}{n}\int_0^{\infty}\min\Bigl\{\overline{\mu}\bigl(S_{\geq
t}(f)\bigr),\overline{\mu}\bigl(-S_{\geq t}(f)\bigr)\Bigr\}\,\dlat t.
\end{equation}
In particular, by choosing $\dlat\mu(x)=\dlat x$, the Lebesgue measure, we
get
\[
\int_{\R^n}\Delta_{-\infty}f(x)\,\dlat x \leq
\binom{2n}{n}\int_{\R^n}f(x)\,\dlat x.
\]
\end{theorem}

\begin{proof}
The proof follows the general ideas of those of \cite[Theorems~4.3
and~4.5]{Co}. Using Fubini's theorem, together with (i) in \eqref{e:superlevel_quasi}, we may write
\[
\Delta_{-\infty}f(x)=\int_0^{\infty}\chi_{_{S_{>t}(f)-S_{>t}(f)}}(x)
\,\dlat t
\]
and, consequently,
\begin{equation}\label{e:delta-inf}
\begin{split}
\int_{\R^n}\Delta_{-\infty}f(x)\,\dlat\mu(x) &
=\int_{\R^n}\int_0^{\infty}\chi_{_{S_{>t}(f)-S_{>t}(f)}}(x)\,\dlat t\,\dlat\mu(x)\\
& \leq\int_0^{\infty}\mu\bigl(S_{\geq t}(f)-S_{\geq
t}(f)\bigr)\,\dlat t.
\end{split}
\end{equation}
Since $f$ is quasi-concave and integrable, the closure of the superlevel
sets $S_{\geq t}(f)$ are convex bodies for all $0<t<\|f\|_{\infty}$. Thus,
we may apply \eqref{e:RS_measures_rad_decreasing} to $S_{\geq t}(f)$
(since the boundary of a convex set has null measure) which, together with
\eqref{e:delta-inf}, allows us to obtain \eqref{e:pquasitheorem}.

Now we note that, if $\dlat\mu(x)=\,\dlat x$, then we  have
\[
\min\Bigl\{\overline{\vol}\bigl(S_{\geq
t}(f)\bigr),\overline{\vol}\bigl(-S_{\geq
t}(f)\bigr)\Bigr\}=\vol\bigl(S_{\geq t}(f)\bigr),
\]
which completes the proof.
\end{proof}

Given a $p$-concave function $f:\R^n\longrightarrow[0,\infty)$, for $p\in
[-\infty,0)$, one can define the $p$-difference of $f$, which remains
$p$-concave (cf. \cite[Proposition~4.2]{Co}), by
\[
\Delta_pf(z)=\sup_{z=x+y}\bigl(f(x)^p+f(-y)^p\bigr)^{1/p}
=\sup_{z=x-y}\bigl(f(x)^p+f(y)^p\bigr)^{1/p}.
\]
where the case $p=-\infty$ is understood as the minimum between both
values. 

\medskip

Theorem \ref{t:pquasitheorem} can be established for any $p\in(-\infty,0)$. It suffices to
note that if $f$ is $p$-concave then it is also quasi-concave, and then,
we may apply inequality \eqref{e:pquasitheorem} 
for $p=-\infty$ together
with the fact that $(a^p+b^p)^{1/p}\leq \min\{a,b\}$ for each $a,b\geq 0$.

\newpage

\noindent
Hence $\Delta_{p}f\leq \Delta_{-\infty}f$.

\begin{remark}\label{r:functional_NO_RS}
As mentioned before, Theorem \ref{t:pquasitheorem} is an application of
Theorem \ref{t:RS_measures_rad_decreasing}. It is a natural and
interesting question whether \eqref{e:RS_measures_rad_decreasing} could be
directly derived from previous functional versions as \eqref{e:Colesanti}.
Just considering $\chi_{_K}\,\phi$ this is not possible because of item
(i) in \eqref{e:superlevel_quasi}: the integral of $\Delta_{-\infty}f$
does not provide (in general) the measure of $K-K$ with respect to the
density $\phi$.
\end{remark}

\subsection{Rogers-Shephard type inequalities for  $CK$ and $\conv\bigl(K\cup(-K)\bigr)$ and their functional versions}

Now we prove the corresponding Rogers-Shephard type inequalities for $CK$
and $\conv\bigl(K\cup(-K)\bigr)$, as well as their equality cases.
\begin{proof}[Proof of Theorem \ref{t:RS_CK_conv_hull_rad_dec}]
Let $f:\R^n\times[0,1]\longrightarrow[0,\infty)$ be the function given~by
\[
f(x,\theta)=\vol\Bigl(\bigl((1-\theta)K\bigr)\cap(x+\theta K)\Bigr).
\]
Note that $f$ is $(1/n)$-concave by \eqref{e:B-M_ineq}, and $\supp f=CK$.
On the one hand, taking the measure $\mu_{n+1}$ on $\R^{n+1}$ given by
$\dlat\mu_{n+1}(x,\theta)=\phi(x)\,\dlat x \,\dlat \theta$, Fubini's
theorem and the change of variable $z=(1-\theta)y$ yield
\begin{equation}\label{e:proving_RS_CK_fubini}
\begin{split}
\int_{CK}f(x,\theta)\,\dlat\mu_{n+1}(x,\theta) &
    =\int_0^1\int_{\R^n}\vol\Bigl(\bigl((1-\theta)K\bigr)\cap(x+\theta K)\Bigr)\phi(x)\,\dlat x\,\dlat\theta\\
 & =\int_0^1\int_{\R^n}\int_{\R^n}\chi_{_{(1-\theta)K}}(z)\chi_{_{x+\theta K}}(z)\,\phi(x)\,\dlat z\,\dlat x\,\dlat\theta\\
 & =\int_0^1\int_{(1-\theta)K}\int_{\R^n}\chi_{_{z-\theta K}}(x)\,\phi(x)\,\dlat x\,\dlat z\,\dlat\theta\\
 & =\int_0^1(1-\theta)^n\int_{K}\mu\bigl((1-\theta)y-\theta K\bigr)\dlat y\,\dlat\theta\\
 & \leq \vol(K)\int_0^1(1-\theta)^n\theta^n\,\dlat\theta
    \sup_{\substack{y\in K\\\theta\in(0,1]}}\frac{\mu\bigl((1-\theta)y-\theta K\bigr)}{\theta^n}\\
 & =\frac{1}{\binom{2n+1}{n}}\frac{\vol(K)}{n+1}
    \sup_{\substack{y\in K\\\theta\in(0,1]}}\frac{\mu\bigl((1-\theta)y-\theta K\bigr)}{\theta^n}.
\end{split}
\end{equation}
Now we define the function $g:CK\longrightarrow[0,\infty)$ given by
\[
g(x,\theta)=f\left(0,\frac{1}{2}\right)\left[1-\frac{\left|(x,\theta)-\bigl(0,\frac{1}{2}\bigr)\right|}%
{\rho_{_{\!CK-(0,\frac{1}{2})}}\Bigl(\bigl((x,\theta)-(0,\frac{1}{2})\bigr)/\bigl|(x,\theta)-(0,\frac{1}{2})\bigr|\Bigr)}
\right]^n,
\]
for every $(x,\theta)\neq (0,1/2)$ and $g(0,1/2)=f(0,1/2)=\vol(K)/2^n$.
Since $f^{1/n}$ is concave, then $f^{1/n}\geq g^{1/n}$ on
$\left[(0,1/2),(0,1/2)+\rho_{_{\!CK-(0,\frac{1}{2})}}(u)u\right]$, and so,
via the polar coordinates $(x,\theta')=(x,\theta)-(0,1/2)=ru$, we get

\begin{equation*}\label{e:proving_RS_CK_polar}
\begin{split}
\int_{CK}f(x, & \theta)\,\dlat\mu_{n+1}(x,\theta) =\int_{CK-(0,\frac{1}{2})} f\left(x,\theta'+\frac{1}{2}\right) \phi(x)\,\dlat x\,\dlat\theta'\\
 & =\int_{\s^{n}}\int_0^{\rho_{_{\!CK-(0,\frac{1}{2})}}(u)} r^{n}f\left(\Bigl(0,\frac{1}{2}\Bigr)+ru\right)\phi\bigl(r P_Hu\bigr)\,\dlat r\,\dlat\sigma(u)\\
 & \geq f\left(0,\frac{1}{2}\right)\int_{\s^{n}}\int_0^{\rho_{_{\!CK-(0,\frac{1}{2})}}(u)}
    \left(1-\frac{r}{\rho_{_{\!CK-(0,\frac{1}{2})}}(u)}\right)^n r^{n}\phi\bigl(r P_Hu\bigr)\,\dlat r \,\dlat\sigma(u),
\end{split}
\end{equation*}
where $H=\bigl\{(x,\theta)\in\R^{n+1}:\theta=0\bigr\}$. Then, Lemma
\ref{l:F_nonposit} yields
\begin{equation}\label{e:proving_RS_CK_polar_lemma}
\begin{split}
\int_{CK}f(x,\theta)\,\dlat\mu_{n+1}(x,\theta) &
 \geq\frac{f\left(0,\frac{1}{2}\right)}{\binom{2n+1}{n}}\int_{\s^n}\int_0^{\rho_{_{\!CK-(0,\frac{1}{2})}}(u)}r^n
    \phi\bigl(r P_Hu\bigr)\,\dlat r\,\dlat\sigma(u)\\
 & =\frac{1}{\binom{2n+1}{n}}\frac{\vol(K)}{2^n}\mu_{n+1}(CK),
\end{split}
\end{equation}
which, together with \eqref{e:proving_RS_CK_fubini}, gives
\eqref{e:R-S_mu_CK}.

Finally we notice that equality holds in \eqref{e:R-S_mu_CK} only if there
is equality in \eqref{e:proving_RS_CK_polar_lemma}. This implies, by Lemma
\ref{l:F_nonposit}, that $\phi\bigl(r P_Hu\bigr)$ is constant on
$\left(0,\rho_{_{\!CK-(0,\frac{1}{2})}}(u)\right)$ for $\sigma$-almost
every $u\in\s^{n}$. Since $\phi$ is continuous at the origin, $\mu_{n+1}$
is a constant multiple of the Lebesgue measure on $CK$ and hence $\mu$ is
so on $P_H(CK)=\conv\bigl(K\cup(-K)\bigr)$ because $\mu_{n+1}$ is a
product measure. Since $(1-\theta)y-\theta K\subset CK$ for all $y\in K$
and any $\theta\in[0,1]$, there is equality in \eqref{e:RS_CK} and
therefore, by Theorem \ref{t:RS_CK}, $K$ is a simplex. The converse is a
direct consequence of Theorem \ref{t:RS_CK}.

Now we prove \eqref{e:R-S_mu_conv}. Note that
$P_H\Bigl(CK\cap\bigl(\c_t(\phi)\times[0,1]\bigr)\Bigr)=\conv\bigl(K\cup(-K)\bigr)\cap\c_t(\phi)$
and, since $0\in K$, then
$CK\cap\bigl(\c_t(\phi)\times[0,1]\bigr)\cap H^{\bot}=[0,1]$.
Hence, Theorem \ref{t:RS_section_proy} yields
$(n+1)\vol_{n+1}\Bigl(CK\cap\bigl(\c_t(\phi)\times[0,1]\bigr)\Bigr)\geq
\vol\Bigl(\conv\bigl(K\cup(-K)\bigr)\cap\c_t(\phi)\Bigr)$,
which, together with Fubini's theorem, gives
\begin{equation*}
\begin{split}
\mu_{n+1}(CK)
& =\|\phi\|_{\infty}\int_{CK}\int_0^1\chi_{_{\c_t(\phi)}}(x)\,\dlat t\,\dlat x\,\dlat\theta\\
 & =\|\phi\|_{\infty}\int_0^1\int_{CK}\chi_{_{\c_t(\phi)\times[0,1]}}(x,\theta)\,\dlat x \,\dlat\theta\,\dlat t\\
 & =\|\phi\|_{\infty}\int_0^1\vol_{n+1}\Bigl(CK\cap\bigl(\c_t(\phi)\times[0,1]\bigr)\Bigr)\,\dlat t\\
 & \geq\|\phi\|_{\infty}\frac{1}{n+1}\int_0^1\vol\Bigl(\conv\bigl(K\cup(-K)\bigr)\cap\c_t(\phi)\Bigr)\,\dlat t\\
 & =\|\phi\|_{\infty}\frac{1}{n+1}\int_0^1\int_{\conv(K\cup(-K))}\chi_{_{\c_t(\phi)}}(x)\,\dlat x \, \dlat t\\
 & =\|\phi\|_{\infty}\frac{1}{n+1}\int_{\conv(K\cup(-K))}\int_0^\frac{\phi(x)}{\|\phi\|_{\infty}}\dlat t\,\dlat x\\
 & =\frac{1}{n+1}\int_{\conv(K\cup(-K))}\phi(x)\,\dlat x
 =\frac{\mu\Bigl(\conv\bigl(K\cup(-K)\bigr)\Bigr)}{n+1}.
\end{split}
\end{equation*}
This, together with \eqref{e:R-S_mu_CK}, shows \eqref{e:R-S_mu_conv}.
Equality in \eqref{e:R-S_mu_conv} implies, in particular, equality in
\eqref{e:R-S_mu_CK} and thus $\mu$ is a constant multiple of the Lebesgue
measure on $\conv\bigl(K\cup(-K)\bigr)$. The proof is now concluded from
the equality case of \eqref{e:RS_conv}.
\end{proof}

\begin{remark}
Taking the function
$f(x,\theta)=\vol\Bigl(\bigl((1-\theta)K\bigr)\cap\bigl(x+\theta(-L)\bigr)\Bigr)$,
and arguing as in the proof of Theorem \ref{t:RS_CK_conv_hull_rad_dec}, an
analogous result can be obtained for two arbitrary convex bodies instead
of $K$ and $-K$. Thus, if $K,L\in\K^n$ contain the origin and $\mu$ is a
measure on $\R^n$ given by $\dlat\mu(x)=\phi(x)\,\dlat x$, where
$\phi:\R^n\longrightarrow[0,\infty)$ is a radially decreasing function,
then
\begin{equation*}\label{e:R-S_CK_K,L}
\begin{split}
\frac{\mu\bigl(\conv(K\cup
L)\bigr)}{n+1}
&\leq\int_0^1\mu\bigl((1-\theta)K+\theta L\bigr)\,\dlat\theta\\
&\leq\frac{2^n}{n+1}\dfrac{\vol(K)}{\vol\bigl(K\cap(-L)\bigr)}\,\sup_{\substack{y\in
K\theta\in(0,1]}}\frac{\mu\bigl((1-\theta)y+\theta L\bigr)}{\theta^n}.
\end{split}
\end{equation*}
\end{remark}

As a consequence of Theorem \ref{t:RS_CK_conv_hull_rad_dec}, we get in
Theorem \ref{t:functional_CK_conv_hull} below functional versions of both
\eqref{e:R-S_mu_CK} and \eqref{e:R-S_mu_conv}. Regarding another
functional version of \eqref{e:RS_conv}, in the log-concave setting, we
refer the reader to \cite[Theorem~1.1]{Co}. The advantage of the
inequality we present here is that, in contrast to the above-mentioned
result, inequality \eqref{e:RS_conv} may recovered just by taking
$f=\chi_{_K}$. We use here the same notation as for
Theorem~\ref{t:pquasitheorem}.

\begin{theorem}\label{t:functional_CK_conv_hull}
Let $f:\R^n\longrightarrow[0,\infty)$ be an integrable quasi-concave
function. Let $\mu$ be a measure on $\R^n$ given by
$\dlat\mu(x)=\phi(x)\,\dlat x$, where $\phi:\R^n\longrightarrow[0,\infty)$
is radially decreasing. Then
\begin{equation}\label{e:functional_CK}
\int_0^1\int_{\R^n}\Delta_{-\infty,\theta}f(x)\,\dlat\mu(x)\,\dlat\theta
\leq\frac{2^n}{n+1}\int_0^{\infty}\sup_{\substack{y\in S_{\geq
t}(f)\\\theta\in(0,1]}}\frac{\mu\bigl((1-\theta)y-\theta S_{\geq
t}(f)\bigr)}{\theta^n}\,\dlat t
\end{equation}
and
\begin{equation}\label{e:functional_conv}
\int_{\R^n}\widetilde{\Delta}_{-\infty}f(x)\,\dlat\mu(x)\leq
2^n\int_0^{\infty}\sup_{\substack{y\in S_{\geq
t}(f)\\\theta\in(0,1]}}\frac{\mu\bigl((1-\theta)y-\theta S_{\geq
t}(f)\bigr)}{\theta^n}\,\dlat t.
\end{equation}
In particular, by choosing $\dlat\mu(x)=\dlat x$, the Lebesgue measure, we
get
\[
\int_0^1\int_{\R^n}\Delta_{-\infty,\theta}f(x)\,\dlat x\,\dlat\theta
\leq\frac{2^n}{n+1}\int_{\R^n}f(x)\,\dlat x
\]
and
\[
\int_{\R^n} \widetilde{\Delta}_{-\infty} f(x)\,\dlat x\leq 2^n \int_{\R^n}
f(x) \,\dlat x.
\]
\end{theorem}

\begin{proof}
Since $f$ is quasi-concave and integrable, the closure of the superlevel
sets $S_{\geq t}(f)$ are convex bodies for all $0<t<\|f\|_\infty$. Thus,
we may apply Theorem \ref{t:RS_CK_conv_hull_rad_dec} to $S_{\geq t}(f)$
(since the boundary of a convex set has null measure) to obtain
\begin{equation*}
\int_0^1\mu\bigl((1-\theta)S_{>t}(f)-\theta S_{>t}(f)\bigr)\,\dlat\theta
\leq \frac{2^n}{n+1}\,\sup_{\substack{y\in S_{\geq
t}(f)\\\theta\in(0,1]}}\frac{\mu\bigl((1-\theta)y-\theta S_{\geq
t}(f)\bigr)}{\theta^n}
\end{equation*}
and
\begin{equation*}
\mu\Bigl(\conv\bigl(S_{>t}(f)\cup(-S_{>t}(f))\bigr)\Bigr)\leq
2^n\,\sup_{\substack{y\in S_{\geq
t}(f)\\\theta\in(0,1]}}\frac{\mu\bigl((1-\theta)y-\theta S_{\geq
t}(f)\bigr)}{\theta^n}.
\end{equation*}
Integrating on $t\in[0,\infty)$, \eqref{e:functional_CK} and
\eqref{e:functional_conv} now follow by applying Fubini's theorem together
with (ii) and (iii) in \eqref{e:superlevel_quasi}, respectively. Finally,
if $\dlat\mu(x)=\,\dlat x$, then we have
\[
\sup_{\substack{y\in S_{\geq
t}(f)\\\theta\in(0,1]}}\frac{\vol\bigl((1-\theta)y-\theta S_{\geq
t}(f)\bigr)}{\theta^n}=\vol\bigl(S_{\geq t}(f)\bigr).
\]
This concludes the proof.
\end{proof}

\section{A projection-section inequality for quasi-concave functions}\label{s:functions}

We start this section by showing a general result for functions that will
be exploited throughout the rest of the paper.

\begin{proposition}\label{l:integral ineqs}
Let $\mu$ be a measure on $\R^n$ given by $\dlat\mu(x)=\phi(x)\,\dlat x$,
where $\phi:\R^n\longrightarrow[0,\infty)$ is quasi-concave and such that
$\|\phi\|_{\infty}=\phi(0)$. Let $f:\R^n\longrightarrow[0,\infty)$ be a
$p$-concave function, $p>0$, with $\|f\|_{\infty}=f(0)$, and let
$g:\R^n\longrightarrow[0,\infty)$ be a measurable function. Then
\begin{equation}\label{e:int(suppf_C_theta)_final}
\int_{\supp f}\int_0^1(1-\theta^{p})^n
    g\bigl((1-\theta^{p})x\bigl)\,\dlat \theta\,\dlat\mu(x)
\leq\dfrac{1}{\|f\|_{\infty}}\int_{\supp f}g(x)f(x)\,\dlat\mu(x).
\end{equation}
Moreover, if $\supp f$ is bounded, $g$ is non-zero on $\supp f$ and $\phi$
is continuous at the origin, equality in
\eqref{e:int(suppf_C_theta)_final} implies that $\mu$ is a constant
multiple of the Lebesgue measure on $\supp f$.
\end{proposition}

\begin{proof}
Since $f$ is $p$-concave, then $\c_{\theta}(f)$ is a convex set for every
$\theta\in[0,1]$. We notice  that
\begin{equation*}\label{e:C_theta1_C_theta2}
\dfrac{\c_{\theta_1}(f)}{1-\theta_1^p}\subset\dfrac{\c_{\theta_2}(f)}{1-\theta_2^p}
\end{equation*}
 for
$0\leq\theta_1\leq\theta_2<1$. In particular, taking $\theta_1=0$, we have
\begin{equation}\label{e:suppf_C_theta}
\supp f\subset\frac{1}{1-\theta^p}\c_{\theta}(f)\quad\text{ for any }\;
\theta\in[0,1),
\end{equation}
and hence
\begin{equation}\label{e:key_inclusion_for_eq}
(\supp
f)\cap\c_t(\phi)\subset\left(\frac{1}{1-\theta^p}\c_{\theta}(f)\right)\cap\c_t(\phi)
\subset\frac{\c_{\theta}(f)\cap\c_t(\phi)}{1-\theta^p}
\end{equation}
for all $\theta\in[0,1)$ and every $t\in[0,1]$. Therefore
\[
\bigl(1-\theta^p\bigr)\bigl[(\supp f)\cap\c_t(\phi)\bigr]
\subset\c_{\theta}(f)\cap\c_t(\phi),
\]
which yields
\begin{equation}\label{e:int(suppf_C_theta)}
\int_0^1\int_0^1\int_{(1-\theta^p)[(\supp f)\cap\c_t(\phi)]}g(x)\,\dlat
x\,\dlat\theta\,\dlat t
\leq\int_0^1\int_0^1\int_{\c_{\theta}(f)\cap\c_t(\phi)}g(x)\,\dlat
x\,\dlat\theta\,\dlat t.
\end{equation}
Now we compute both sides of inequality (\ref{e:int(suppf_C_theta)}). On
the one hand, by Fubini's theorem and the change of variable
$x=\bigl(1-\theta^p\bigr)y$, we get
\[
\begin{split}
\int_0^1\int_0^1 & \int_{(1-\theta^p)[(\supp f)\cap\c_t(\phi)]}g(x)\,\dlat x\,\dlat\theta\,\dlat t \\
& =\int_0^1\int_0^1\int_{(\supp f)\cap\c_t(\phi)}g\bigl((1-\theta^p)y\bigr)
    (1-\theta^p)^n\,\dlat y\,\dlat\theta\,\dlat t\\
 & =\int_{\supp f}\int_0^1(1-\theta^p)^ng\bigl((1-\theta^p)y\bigl)
    \int_0^1\chi_{_{\c_t(\phi)}}(y)\,\dlat t\,\dlat\theta\,\dlat y\\
 & =\int_{\supp f}\int_0^1(1-\theta^p)^ng\bigl((1-\theta^p)y\bigl)
    \frac{\phi(y)}{\|\phi\|_{\infty}}\,\dlat\theta\,\dlat y\\
 & =\dfrac{1}{\|\phi\|_{\infty}}\int_{\supp f}\int_0^1(1-\theta^p)^n
    g\bigl((1-\theta^p)y\bigl)\,\dlat\theta\,\dlat\mu(y).
\end{split}
\]
On the other hand, using again Fubini's theorem,
\[
\begin{split}
\int_0^1\int_0^1\int_{\c_{\theta}(f)\cap\c_t(\phi)}g(x)
    \,\dlat x\,\dlat\theta\,\dlat t
 & =\int_0^1\int_0^1\int_{\R^n}g(x)\,\chi_{_{\c_{\theta}(f)}}(x)\chi_{_{\c_t(\phi)}}(x)\,\dlat x\,\dlat\theta\,\dlat t\\
 & =\int_{\R^n}g(x)\int_0^1\chi_{_{\c_t(\phi)}}(x)\int_0^1\chi_{_{\c_{\theta}(f)}}(x)\,\dlat\theta\,\dlat t\,\dlat x\\
 & =\int_{\supp f}g(x)\dfrac{f(x)}{\|f\|_{\infty}}\dfrac{\phi(x)}{\|\phi\|_{\infty}}\,\dlat x\\
 & =\dfrac{1}{\|f\|_{\infty}\|\phi\|_{\infty}}\int_{\supp f}g(x)\,f(x)\,\dlat\mu(x).
\end{split}
\]
Thus, \eqref{e:int(suppf_C_theta)_final} follows from inequality
(\ref{e:int(suppf_C_theta)}).

Now we deal with the equality case. First we observe that since $\supp f$
is a bounded set and $f$ is $p$-concave, then $\c_{\theta}(f)$ is a
bounded convex set for all $\theta\in[0,1)$.

Without loss of generality we may assume that $\phi$ is upper
semicontinuous. Indeed, otherwise we would work with its upper closure,
which is determined via the closure of the superlevel sets of $\phi$ (see
\cite[page~14 and Theorem~1.6]{RoWe}) and thus defines the same measure
because of Fubini's theorem together with the facts that all the
superlevel sets of $\phi$ are convex (since it is quasi-concave) and the
boundary of a convex set has null (Lebesgue) measure. Then its superlevel
sets $\c_t(\phi)$ are closed (cf. \cite[Theorem~1.6]{RoWe}) for every
$t\in[0,1]$. In the same way, $f$ may be assumed to be upper
semicontinuous (in fact, it is already continuous in the interior of its
support, because of the $p$-concavity). Moreover, since the definitions of
both $\c_\theta(f)$ and $\c_t(\phi)$ involve the essential supremum, these
superlevel sets have positive volume for all $\theta<1$ and $t<1$, and
therefore both $\c_\theta(f)$ and $\c_t(\phi)$ are closed convex sets with
non-empty interior, for any $\theta,t\in[0,1)$. From the continuity of
$\phi$ at the origin, we know that $0\in\inter\c_t(\phi)$ for all $t<1$
and then $0\in\c_\theta(f)\cap\inter\c_t(\phi)$ because
$f(0)=\|f\|_{\infty}$. Hence, and taking into account that $\supp f$ (and
thus $\c_\theta(f)$ for any $\theta\in[0,1]$) is bounded, both
$\c_\theta(f)\cap(1-\theta^{p})\c_t(\phi)$ and
$\c_\theta(f)\cap\c_t(\phi)$ are convex bodies for all $\theta,t\in[0,1)$.

Thus, if equality holds in \eqref{e:int(suppf_C_theta)_final} then, in
particular, there is equality in the right-hand inclusion of
\eqref{e:key_inclusion_for_eq} for almost all $\theta\in[0,1]$ and almost
all $t\in[0,1]$, because $g>0$ on $\supp f$.

Let us assume that there exists $x_0\in\supp f$ such that
$\phi(x_0)<\|\phi\|_{\infty}$. Taking
$t\in\bigl(\phi(x_0)/\|\phi\|_{\infty},1\bigr]$, since
$x_0\not\in\c_t(\phi)$ then we have that
\[
(\supp f)\cap\c_t(\phi)\subsetneq\supp f.
\]
Let $x_t\in\bd\bigl((\supp f)\cap\c_t(\phi)\bigr)\backslash\bd(\supp f)$.
Since both sets are convex bodies, we can always take $x_t\neq 0$. Then
for all $t\in\bigl(\phi(x_0)/\|\phi\|_{\infty},1\bigr]$, the continuity of
$f$ on $\inter(\supp f)$ yields the existence of $\theta_t\in(0,1)$ such
that
\[
x_t\in\c_{\theta}(f)\cap\c_t(\phi)\quad\text{ for all }\;
\theta\in[0,\theta_t).
\]
However, since $x_t\in\bd\c_t(\phi)$ and $0\in\inter\c_t(\phi)$,
\[
x_t\not\in\c_{\theta}(f)\cap\bigl(1-\theta^p\bigr)\c_t(\phi).
\]
This contradicts the equality in the right-hand inclusion of
\eqref{e:key_inclusion_for_eq} for almost every $\theta\in[0,1]$ and
$t\in[0,1]$.

Therefore we may conclude that $\phi(x)\geq\|\phi\|_{\infty}$ for all
$x\in\supp f$ and thus $\phi\equiv\|\phi\|_{\infty}$ almost everywhere on
$\supp f$. This implies that $\mu$ is a constant multiple of the Lebesgue
measure on $\supp f$.
\end{proof}

It is an interesting question whether Proposition \ref{l:integral ineqs}
can be adapted to log-concave functions, i.e., when $p=0$. We notice that
the above approach cannot be followed in this case. Indeed, considering
e.g. the function $f:\R\longrightarrow[0,\infty)$ given by
$f(x)=e^{-x^2}$, we have that $\supp f=\R$ whereas $\c_{\theta}(f)$ is a
convex body for all $t\in(0,1]$. Hence, there is no chance to get an
inclusion of the type \eqref{e:suppf_C_theta}, i.e., $\lambda(\theta)\supp
f\subset\c_{\theta}(f)$ for any $\theta\in[0,1]$ and some
$\lambda(\theta)>0$.

\smallskip

In what follows we use Proposition \ref{l:integral ineqs} to prove several
results, including Theorem \ref{t:functional_RS}. Let us first introduce a
helpful family of constants and notice a few facts. We denote by
\[
\alpha^n_{p,q}=\int_0^1(1-\theta^p)^n\,\theta^{p\,q}\,\dlat\theta
    =\dfrac{\Gamma\left(\frac{1}{p}+q\right)\Gamma(1+n)}{p\,\Gamma\left(1+n+\frac{1}{p}+q\right)},
\]
for each $p,q>0$. Let us assume that $g$ is concave. Then
\[
g\bigl((1-\theta^p)x\bigr)\geq\theta^pg(0)+(1-\theta^p)g(x),
\]
and so, we get from \eqref{e:int(suppf_C_theta)_final} that
\begin{equation}\label{e:int(suppf_C_theta)_g_concave}
\alpha^n_{p,1}\,g(0)\,\mu(\supp f)+\alpha^{n+1}_{p,0}\int_{\supp
f}g(x)\,\dlat\mu(x)\leq\dfrac{1}{\|f\|_{\infty}}\int_{\supp
f}g(x)\,f(x)\,\dlat\mu(x).
\end{equation}
Another possibility is assuming that $g$ is radially decreasing. Then,
from \eqref{e:int(suppf_C_theta)_final}, we get
\begin{equation}\label{e:int(suppf_C_theta)_g_rad_dec}
\alpha^n_{p,0}\int_{\supp f}g(x)\,\dlat\mu(x)
\leq\dfrac{1}{\|f\|_{\infty}}\int_{\supp f}g(x)\,f(x)\,\dlat\mu(x).
\end{equation}
We point out that $\alpha^n_{p,0}=\alpha^n_{p,1}+\alpha^{n+1}_{p,0}$,
which shows that the expression on the left-hand side of
\eqref{e:int(suppf_C_theta)_g_concave} and that of
\eqref{e:int(suppf_C_theta)_g_rad_dec} are in a sense ``similar'', as
shown by considering the constant function $g(x)=1$. Indeed, when $g\equiv
1$, \eqref{e:int(suppf_C_theta)_g_rad_dec} reads
\begin{equation}\label{e:int(suppf_C_theta)_g_uno}
\alpha^n_{p,0}\,\mu(\supp f)\leq\dfrac{1}{\|f\|_{\infty}}\int_{\supp
f}f(x)\,\dlat\mu(x).
\end{equation}
Moreover, it can be proved that \eqref{e:int(suppf_C_theta)_g_uno} remains
true even in the more general case when $\|f\|_{\infty}=f(x_0)$ for an
arbitrary $x_0\in\R^n$, and without the maximality assumption for $\phi$.

\begin{corollary}\label{c:|f|_x0}
Let $f:\R^n\longrightarrow[0,\infty)$ be a $p$-concave function, $p>0$,
with $\|f\|_{\infty}=f(x_0)$ for some $x_0\in\R^n$, and let $\mu$ be a
measure on $\R^n$ given by $\dlat\mu(x)=\phi(x)\,\dlat x$, where
$\phi:\R^n\longrightarrow[0,\infty)$ is a bounded quasi-concave function.
Then
\begin{equation}\label{e:mu(suppf)_x0}
\alpha^n_{p,0}\frac{\phi(x_0)}{\|\phi\|_{\infty}}\,\mu(\supp
f)\leq\dfrac{1}{\|f\|_{\infty}}\int_{\supp f}f(x)\,\dlat\mu(x).
\end{equation}
Moreover, if $\supp f$ is bounded and $\phi$ is continuous at $x_0$,
equality in \eqref{e:mu(suppf)_x0} implies that $\mu$ is a constant
multiple of the Lebesgue measure on $\supp f$.
\end{corollary}

\begin{proof}
The proof follows similar steps as those of Proposition \ref{l:integral
ineqs}, but with some key variations. We will highlight these differences.

We consider the function $\psi:\R^n\longrightarrow[0,\infty)$ given by
$\psi(x)=f(x+x_0)$, which satisfies $\|\psi\|_{\infty}=\|f\|_{\infty}$ and
$\supp\psi=(\supp f)-x_0$. Then (cf. \eqref{e:suppf_C_theta})
\begin{equation}\label{e:suppf_C_theta2}
\supp\psi\subset\frac{1}{1-\theta^p}\c_{\theta}(\psi)\quad\text{ for all
}\; \theta\in[0,1).
\end{equation}
We observe that $y\in\c_{\theta}(\psi)$ if and only if $f(y+x_0)\geq
\theta\|f\|_{\infty}$, or equivalently, when $y+x_0\in\c_{\theta}(f)$. Hence,
$\c_{\theta}(\psi)+x_0=\c_{\theta}(f)$, and thus \eqref{e:suppf_C_theta2}
turns into
\[
(\supp
f)-x_0\subset\frac{1}{1-\theta^p}\bigl(\c_{\theta}(f)-x_0\bigr)\quad\text{
for all }\; \theta\in[0,1).
\]
Therefore
\begin{equation*}\label{e:key_inclusion_for_eq2}
\begin{split}
\bigl((\supp f)-x_0\bigr)\cap\bigl(\c_t(\phi)-x_0\bigr) &
\subset\left(\frac{1}{1-\theta^p}\bigl(\c_{\theta}(f)-x_0\bigr)\right)\cap\bigl(\c_t(\phi)-x_0\bigr)\\
 & \subset\frac{1}{1-\theta^p}\Bigl(\bigl[\c_{\theta}(f)\cap\c_t(\phi)\bigr]-x_0\Bigr)
\end{split}
\end{equation*}
for all $\theta\in[0,1)$ and every
$t\in\bigl[0,\phi(x_0)/\|\phi\|_{\infty}\bigr]$, where in the last
inclusion we have used that $x_0\in\c_t(\phi)$. Consequently, we obtain
\begin{equation}\label{e:incl_I}
(1-\theta^p)\Bigl(\bigl[(\supp
f)\cap\c_t(\phi)\bigr]-x_0\Bigr)\subset\bigl(\c_{\theta}(f)\cap\c_t(\phi)\bigr)-x_0.
\end{equation}
Next, integrating over $x\in\R^n$ the constant function $1$, using
(\ref{e:incl_I}) and the change of variable $x=(1-\theta^p)y$, we
get
\[
(1-\theta^p)^n\int_{[(\supp f)\cap\c_t(\phi)]-x_0}\dlat y
\leq\int_{[\c_{\theta}(f)\cap\c_t(\phi)]-x_0}\dlat y,
\]
which yields
\begin{equation}\label{e:int_supp_Ct<int_Ctheta_Ct}
(1-\theta^p)^n\int_{(\supp f)\cap\c_t(\phi)}\dlat x
\leq\int_{\c_{\theta}(f)\cap\c_t(\phi)}\dlat x.
\end{equation}
Now, computing the left-hand side in \eqref{e:mu(suppf)_x0}, we get
\[
\begin{split}
\alpha^n_{p,0}\frac{\phi(x_0)}{\|\phi\|_{\infty}}\,\mu(\supp f)
 & =\alpha^n_{p,0}\|\phi\|_{\infty}\int_{\supp f}\frac{\phi(x_0)}{\|\phi\|_{\infty}}\,\frac{\phi(x)}{\|\phi\|_{\infty}}\,\dlat x\\
 & \leq\|\phi\|_{\infty}\int_0^1(1-\theta^p)^n\,\dlat\theta
    \int_{\supp f}\min\left\{\frac{\phi(x)}{\|\phi\|_{\infty}},\frac{\phi(x_0)}{\|\phi\|_{\infty}}\right\}\,\dlat x\\
 & =\|\phi\|_{\infty}\int_0^1\int_0^{\frac{\phi(x_0)}{\|\phi\|_{\infty}}}
    (1-\theta^p)^n\int_{(\supp f)\cap\c_t(\phi)} \dlat x\,\dlat t\,\dlat\theta.
\end{split}
\]
Applying \eqref{e:int_supp_Ct<int_Ctheta_Ct} we obtain the desired
inequality. Indeed from the above computation we get
\begin{equation*}
\begin{split}
\alpha^n_{p,0}\frac{\phi(x_0)}{\|\phi\|_{\infty}}\,\mu(\supp f) &
    \leq\|\phi\|_{\infty}\int_0^1\int_0^{\frac{\phi(x_0)}{\|\phi\|_{\infty}}}\int_{\c_{\theta}(f)\cap\c_t(\phi)}\dlat x\,\dlat t\,\dlat \theta\\
 & =\frac{\|\phi\|_{\infty}}{\|f\|_{\infty}}\int_{\supp f}f(x)\int_0^{\frac{\phi(x_0)}{\|\phi\|_{\infty}}}\chi_{_{\c_t(\phi)}}(x)\,\dlat t\,\dlat x\\
 & \leq\frac{\|\phi\|_{\infty}}{\|f\|_{\infty}}\int_{\supp f}f(x)\int_0^1\chi_{_{\c_t(\phi)}}(x)\,\dlat t\,\dlat x\\
 &=\frac{1}{\|f\|_{\infty}}\int_{\supp f}f(x)\,\dlat \mu(x).
\end{split}
\end{equation*}
For the proof of the equality case we observe, on the one hand, that if
equality holds in \eqref{e:mu(suppf)_x0} then, in particular,
\[
\int_{\supp
f}f(x)\int_{\frac{\phi(x_0)}{\|\phi\|_{\infty}}}^1\chi_{_{\c_t(\phi)}}(x)\,\dlat
t\,\dlat x=0,
\]
which yields $\phi(x_0)=\esssup_{x\in \supp f}\phi(x)$.

On the other hand, we may replace $\|\phi\|_{\infty}$ by $\esssup_{x\in
\supp f}\phi(x)$ in the above argument to get also
\begin{equation*}\label{e:remark_lemma}
\alpha_{p,0}^n\frac{\phi(x_0)}{\esssup_{x\in \supp f}\phi(x)}\,\mu(\supp
f) \leq\frac{1}{\|f\|_{\infty}}\int_{\supp f}f(x)\,\dlat \mu(x),
\end{equation*}
and since
\[
\alpha^n_{p,0}\frac{\phi(x_0)}{\|\phi\|_{\infty}}\,\mu(\supp
f)\leq\alpha_{p,0}^n\,\frac{\phi(x_0)}{\esssup_{x\in \supp
f}\phi(x)}\,\mu(\supp f)=\alpha_{p,0}^n\,\mu(\supp f),
\]
equality in \eqref{e:mu(suppf)_x0} implies that
$\phi(x_0)=\|\phi\|_{\infty}$.

Finally, due to the fact that $\phi(x_0)=\|\phi\|_{\infty}$, the rest of
the proof of the equality case is entirely analogous to the one in
Proposition \ref{l:integral ineqs}, and we do not repeat it here.
\end{proof}

As an application of Proposition \ref{l:integral ineqs}, and the
above-mentioned consequences of it, we show Theorem~\ref{t:functional_RS}.

\begin{proof}[Proof of Theorem~\ref{t:functional_RS}]
For all $t\in[0,1]$, the function
$\varphi_t:P_H\c_t(f)\longrightarrow[0,\infty)$ given by
\[
\varphi_t(x)=\vol_k\bigl(\c_t(f)\cap(x+H^{\bot})\bigr)
\]
is ($1/k$)-concave, because of the Brunn-Minkowski inequality
\eqref{e:B-M_ineq}, and $\supp\varphi_t=P_H\c_t(f)$. By hypothesis we have
$\|\varphi_t\|_{\infty}=\varphi_t(0)$. Then, by applying
\eqref{e:int(suppf_C_theta)_g_rad_dec}  to $\varphi_t$, we get
\begin{equation}\label{e:gf_t}
\alpha_{1/k,0}^{n-k}\int_{P_H\c_t(f)}g(x)\,\dlat x
\leq\dfrac{1}{\|\varphi_t\|_{\infty}}\int_{H}g(x)\,\varphi_t(x)\,\dlat x
\end{equation}
and hence, integrating each side of inequality  (\ref{e:gf_t}) over
$t\in[0,1]$ and noticing that $\alpha_{1/k,0}^{n-k}=\binom{n}{k}^{-1}$, it
follows that
\begin{equation}\label{e:gf_ineq_para_phi_t_2}
\int_0^1\int_{P_H\c_t(f)}g(x)\,\dlat x
\int_{H^{\bot}}\chi_{_{\c_t(f)}}(y)\dlat y\,\dlat t
\leq\binom{n}{k}\int_0^1\int_{H}g(x)\int_{x+H^{\bot}}\chi_{_{\c_t(f)}}(y)\dlat
y\,\dlat x\,\dlat t.
\end{equation}
On the one hand, by Fubini's theorem and noticing that
\[
P_H\c_t(f)\supset P_H\Bigl(\bigl\{x\in\R^n: f(x)>
t\|f\|_{\infty}\bigr\}\Bigr)=\bigl\{x\in H: P_Hf(x)>
t\|f\|_{\infty}\bigr\},
\]
we obtain
\begin{equation}\label{e:left}
\begin{split}
\int_0^1\int_{H}g(x)\chi_{_{P_H\c_t(f)}}(x)\,\dlat x & \int_{H^{\bot}}
    \chi_{_{\c_t(f)}}(y)\,\dlat y\,\dlat t\\
 & =\int_{H}\int_{H^{\bot}}g(x)\int_0^1\chi_{_{P_H\c_t(f)}}(x)\chi_{_{\c_t(f)}}(y)\,\dlat t\,\dlat y\,\dlat x\\
 & \geq\int_{H}\int_{H^{\bot}}g(x)\min\left\{\frac{P_Hf(x)}{\|f\|_{\infty}},\frac{f(y)}{\|f\|_{\infty}}\right\}\,\dlat y\,\dlat x\\
 & \geq\int_{H}\int_{H^{\bot}}g(x)\frac{P_Hf(x)}{\|f\|_{\infty}}\frac{f(y)}{\|f\|_{\infty}}\,\dlat y\,\dlat x\\
 & =\int_{H}g(x)\frac{P_Hf(x)}{\|f\|_{\infty}}\,\dlat x\int_{H^{\bot}}\frac{f(y)}{\|f\|_{\infty}}\,\dlat y.
\end{split}
\end{equation}
On the other hand, Fubini's theorem yields
\begin{equation}\label{e:right}
\begin{split}
\int_0^1\int_{H}g(x)\int_{x+H^{\bot}}\chi_{_{\c_t(f)}}(y)\dlat y\,\dlat
    x\,\dlat t & =\int_{H}g(x)\int_{x+H^{\bot}}\int_0^1\chi_{_{\c_t(f)}}(y)\dlat t\,\dlat y\,\dlat x\\
 & =\int_{H}\int_{x+H^{\bot}}g(x)\frac{f(y)}{\|f\|_{\infty}}\,\dlat y\,\dlat x\\
 & =\int_{\R^n}g(P_Hz)\frac{f(z)}{\|f\|_{\infty}}\,\dlat z.
\end{split}
\end{equation}
Therefore, from \eqref{e:gf_ineq_para_phi_t_2}, \eqref{e:left} and
\eqref{e:right} we obtain
\begin{equation*}
\int_{H}g(x)P_Hf(x)\,\dlat x\int_{H^{\bot}}f(y)\,\dlat y
\leq\binom{n}{k}\|f\|_{\infty}\int_{\R^n}g(P_Hx)f(x)\,\dlat x.
\end{equation*}
This concludes the proof.
\end{proof}

With the above approach, but using \eqref{e:int(suppf_C_theta)_g_uno}
instead of \eqref{e:int(suppf_C_theta)_g_rad_dec}, we notice that the
maximality assumption at the origin can be relaxed to get the following
result, which has been recently obtained in the setting of a log-concave
integrable function in \cite[Theorem~1.1]{AAGJV}.
\begin{corollary}\label{c:RS_sect_proj}
Let $k\in\{1,\dots,n-1\}$ and $H\in\G(n,n-k)$. Let
$f:\R^n\longrightarrow[0,\infty)$ be a quasi-concave function such that
\[
\sup_{x\in H}\vol_k\bigl(\c_t(f)\cap\bigl(x+H^{\bot}\bigr)\bigr)
\]
is attained for all $t\in(0,1)$. Then
\begin{equation}\label{e:proy_sect_f}
\int_{H}P_Hf(x)\,\dlat x\max_{x_0\in H}\int_{x_0+H^{\bot}}f(y)\,\dlat y
\leq\binom{n}{k}\|f\|_{\infty}\int_{\R^n}f(x)\,\dlat x.
\end{equation}
\end{corollary}

We point out that, in the case of an integrable function $f$ whose
restriction to its support is continuous, the above assumption on the
volume of the sections of $\c_t(f)$ trivially holds, since $\c_t(f)$ is
compact for every $t\in(0,1)$. Notice also that, when dealing with certain
classes of functions with a more restrictive concavity (such as
log-concave ones), continuity on the interior of their support is already
guaranteed.

\section{Rogers-Shephard type inequalities for measures with quasi-concave densities}\label{s:quasi_concave}

As a direct application of Corollary \ref{c:RS_sect_proj} we obtain the
following result.
\begin{theorem}\label{t:RS_seccion_proy_quasi}
Let $k\in\{1,\dots,n-1\}$ and $H\in\G(n,n-k)$. Let
$\phi_i:\R^i\longrightarrow[0,\infty)$, $i=n-k,k$, be functions with
$\|\phi_i\|_{\infty}=\phi_i(0)$, and such that the function
$\phi:\R^n\longrightarrow[0,\infty)$ given by
$\phi(x,y)=\phi_{n-k}(x)\phi_k(y)$, $x\in\R^{n-k}$, $y\in\R^k$, is
quasi-concave. Let $\mu_n=\mu_{n-k}\times\mu_{k}$ be the product measure
on $\R^n$ given by $\dlat\mu_{n-k}(x)=\phi_{n-k}(x)\,\dlat x$ and
$\dlat\mu_{k}(y)=\phi_k(y)\,\dlat y$. Let $K\in\K^n$ with $P_HK\subset K$
and so that $\vol\bigl(\c_t(\phi)\cap K\cap(x+H^{\bot})\bigr)$ attains its
maximum for all $t\in(0,1)$. Then
\begin{equation}\label{e:RS_seccion_proy_quasi}
\mu_{n-k}\bigl(P_HK\bigr)\max_{x_0\in
H}\left[\frac{\phi_{n-k}(x_0)}{\|\phi_{n-k}\|_{\infty}}\mu_k\bigl(K\cap(x_0+H^{\bot})\bigr)\right]\leq\binom{n}{k}\mu_n(K).
\end{equation}
\end{theorem}

\begin{proof}
It is a straightforward consequence of \eqref{e:proy_sect_f} applied to
the function $f:\R^n\longrightarrow[0,\infty)$ given by
$f(x,y)=\phi_{n-k}(x)\phi_k(y)\chi_{_K}(x,y)$. Indeed, since $P_HK\subset
K$ then
\[
P_Hf(x)=\sup_{y\in
H^{\bot}}\phi_{n-k}(x)\phi_k(y)\chi_{_K}(x,y)=\phi_{n-k}(x)\phi_{k}(0)\chi_{_{P_HK}}(x)
\]
and $\|f\|_{\infty}=\phi_{n-k}(0)\phi_{k}(0)$.
\end{proof}

We point out that the assumption $P_HK\subset K$ is needed in order to
conclude the above Rogers-Shephard type inequality (as well as Theorem
\ref{t:RS_secc_proy_K(0)}):

\begin{example}\label{r:hip_P_HK}
Let $\mu_1$ be the measure on $\R$ given by $\dlat\mu_1(x)=e^{-x^2}\,\dlat
x$ and let $\mu_2=\mu_{1}\times\mu_{1}$, i.e.,
$\dlat\mu_2(x)=e^{-|x|^2}\,\dlat x$. Let
$H=\bigl\{(x,y)\in\R^2:y=0\bigr\}$ and, for a given $0<\alpha<\pi/2$, let
$K_{\alpha}$ be the centrally symmetric parallelogram
$K_{\alpha}=\conv\bigl\{(1,\tan\alpha\pm 1),(-1,-\tan\alpha\pm 1)\bigr\}$.

On the one hand, $K_{\alpha}(0)=\bigl[(0,1),(0,-1)\bigr]$ is the `maximal'
section of $K_{\alpha}$ (with respect to $\mu_1$) and
$P_HK_{\alpha}=\bigl[(-1,0),(1,0)\bigr]$. On the other hand, since
$K_{\alpha}$ is contained in the infinite strip $S_{\alpha}$ determined by
the straight lines $y=(\tan\alpha) x\pm 1$, and $\mu_2$ is rotationally
invariant, we have that
\[
\mu_2(K_{\alpha})\leq\mu_2(S_{\alpha})=\sqrt{2\pi}\,\mu_1(I_{\alpha}),
\]
where $I_{\alpha}$ denotes the line segment centered at the origin and
with length the width of $S_{\alpha}$.

Hence, $\mu_1(I_{\alpha})$, and so $\mu_2(K_{\alpha})$, can be made
arbitrarily small when $\alpha\rightarrow\pi/2$. However, the term
$\mu_1\bigl(P_HK_{\alpha}\bigr)\mu_1\bigl(K_{\alpha}(0)\bigr)=\mu_1\bigl([(-1,0),(1,0)]\bigr)^2$
is a fixed positive constant. This shows the necessity of assuming
$P_HK\subset K$ in order to derive both \eqref{e:RS_seccion_proy_quasi}
and \eqref{e:RS_secc_proy_K(0)}.
\end{example}

In order to avoid the assumption $P_HK\subset K$, one may exchange the
orthogonal projection by the corresponding maximal section. To this end,
first we fix some notation: given a measure $\mu$ in $\R^n$ with density
$\phi$, we will denote by $\mu_i$, $i=1,\dots,n-1$, the {\it marginal} of
$\mu$ in the corresponding $i$-dimensional affine subspace, i.e., for
given $M\subset z+H$ with $H\in\G(n,i)$ and $z\in H^{\bot}$,
\[
\mu_i(M)=\int_H\chi_{_{M}}(x,z)\phi(x,z)\,\dlat x.
\]
Taking the function $f:\R^n\longrightarrow[0,\infty)$ given by
$f(x,y)=\phi(x,y)\chi_{_K}(x,y)$, $x\in H$, $y\in H^{\bot}$, since
\[
P_Hf(x)=\sup_{y\in
H^{\bot}}\phi(x,y)\chi_{_K}(x,y)\geq\phi(x,y)\chi_{_K}(x,y)=f(x,y),
\]
we get the following result, as direct consequence of
\eqref{e:proy_sect_f}.
\begin{corollary}\label{c:coro_RS_sect_sect}
Let $k\in\{1,\dots,n-1\}$ and $H\in\G(n,n-k)$. Let $\mu$ be a measure on
$\R^n$ given by $\dlat\mu(x)=\phi(x)\,\dlat x$, where
$\phi:\R^n\longrightarrow[0,\infty)$ is a quasi-concave function with
$\|\phi\|_{\infty}=\phi(0)$. Let $K\in\K^n$ be such that there exists the
maximum of $\vol\bigl(\c_t(\phi)\cap K\cap(x+H^{\bot})\bigr)$ for all
$t\in(0,1)$. Then
\begin{equation}\label{e:coro_RS_sect_sect}
\max_{y\in H}\mu_{n-k}\bigl(K\cap (y+H)\bigr)\max_{x_0\in
H}\mu_k\bigl(K\cap(x_0+H^{\bot})\bigr)\leq\binom{n}{k}\|\phi\|_{\infty}\mu(K).
\end{equation}
\end{corollary}

We notice that, from \eqref{e:RS_seccion_proy_quasi},
\begin{equation}\label{e:RS_seccion_proy_quasi:2}
\mu_{n-k}\bigl(P_HK\bigr)\mu_k\bigl(K\cap
H^{\bot}\bigr)\leq\binom{n}{k}\mu_n(K)
\end{equation}
holds provided that the density of $\mu_n$,
$\phi(x,y)=\phi_{n-k}(x)\phi_k(y)$, is quasi-concave. Although the latter
implies that both $\phi_{n-k}, \phi_k$ are quasi-concave, the converse is,
in general, not true. In the following we exploit the approach followed in
the previous section in order to derive \eqref{e:RS_seccion_proy_quasi:2}
for the more general case of measures $\mu_{n-k}, \mu_{k}$, with radially
decreasing and quasi-concave densities, respectively, and their product
$\mu_n=\mu_{n-k}\times\mu_{k}$, provided that the maximality assumption
\[
\max_{x\in P_HK}\vol_k\bigl(\c_t(\phi_k)\cap K(x)\bigr)
=\vol_k\bigl(\c_t(\phi_k)\cap K(0)\bigr)
\]
holds. Again, we need to assume the condition $P_HK\subset K$.

\begin{proof}[Proof of Theorem~\ref{t:RS_secc_proy_K(0)}]
By an appropriate choice of the coordinate axes, we may assume that
$H=\{x_{n-k+1}=\dots=x_n=0\}$. For every $t\in[0,1]$, and $x\in P_H K$, we
consider the set
\[
\c_{x,t}=\Bigl(\{0\}\times\c_t(\phi_k)\Bigr)\cap K(x)
\]
and the function $\varphi_t:P_HK\longrightarrow[0,\infty)$ given by
\[
\varphi_t(x)=\vol_k\bigl(\c_{x,t}\bigr).
\]
Since $P_HK\subset K$ and $\phi_k$ is continuous at the origin (which
implies that $0\in\inter\c_t(\phi_k)$ for all $t<1$), we may assure that,
for every $t<1$, $\varphi_t(x)>0$ for any $x$ in the (relative) interior
of $P_HK$ and hence $\supp\varphi_t=P_HK$. Moreover, $\varphi_t$ is
$(1/k)$-concave by \eqref{e:B-M_ineq} and, by hypothesis, we have
$\|\varphi_t\|_{\infty}=\varphi_t(0)$.

Then, applying \eqref{e:int(suppf_C_theta)_g_rad_dec}, with $p=1/k$, to
the function $g:P_HK\longrightarrow[0,\infty)$ given by
$g(x,0)=\phi_{n-k}(x)$, $x\in\R^{n-k}$, we get
\begin{equation}\label{e:secc1}
\int_{P_HK}\phi_{n-k}(x)\,\dlat x
\leq\binom{n}{k}\dfrac{1}{\|\varphi_t\|_{\infty}}\int_{P_HK}\phi_{n-k}(x)\,\varphi_t(x)\,\dlat
x,
\end{equation}
and hence, integrating (\ref{e:secc1}) over $t\in[0,1]$, we obtain
\begin{equation*}\label{e:gf_ineq_para_phi_t_proy_secc_2}
\int_0^1\int_{P_HK}\phi_{n-k}(x)\,\dlat x\int_{\R^k}
\chi_{_{\c_{0,t}}}(y)\,\dlat y\,\dlat
t\leq\binom{n}{k}\int_0^1\int_{P_HK}\phi_{n-k}(x)\int_{\R^k}\chi_{_{\c_{x,t}}}(y)
\,\dlat y\,\dlat x\,\dlat t.
\end{equation*}
Therefore, by Fubini's theorem we have
\begin{equation*}\label{e:left_proy_secc}
\begin{split}
\mu_{n-k}\bigl(P_HK\bigr)\mu_k\bigl(K\cap H^{\bot}\bigr) &
    =\|\phi_k\|_{\infty}\int_{P_HK}\phi_{n-k}(x)\,\dlat x\int_{K(0)}\int_0^1 \chi_{_{\c_t(\phi_k)}}(y)\,\dlat t\,\dlat y\\
 & =\|\phi_k\|_{\infty}\int_0^1\int_{P_HK}\phi_{n-k}(x)\,\dlat x\int_{\R^k} \chi_{_{\c_{0,t}}}(y)\,\dlat y\,\dlat t\\
 & \leq\binom{n}{k}\|\phi_k\|_{\infty}\int_0^1\int_{P_HK}\phi_{n-k}(x)\int_{\R^k}\chi_{_{\c_{x,t}}}(y)\,\dlat y\,\dlat x\,\dlat t\\
 & =\binom{n}{k}\|\phi_k\|_{\infty}\int_{P_HK}\phi_{n-k}(x)\int_{K(x)}\int_0^1\chi_{_{\c_t(\phi_k)}}(y)\,\dlat t\,\dlat y\,\dlat x\\
 & =\binom{n}{k}\int_{P_HK}\phi_{n-k}(x)\mu_k\bigl(K(x)\bigr)\dlat x=\binom{n}{k}\mu_n(K).
\end{split}
\end{equation*}
This concludes the proof.
\end{proof}

Next we show an extension of the above Rogers-Shephard type inequalities
involving maximal sections of convex bodies (cf.
\eqref{e:coro_RS_sect_sect}) in the spirit of \cite[Lemma~4.1]{AAGJV}.
\begin{corollary}\label{c:RS_E_H}
Let $i,j\in\{2,\dots,n-1\}$, $i+j\geq n+1$, and let $E\in\G(n,i)$,
$H\in\G(n,j)$ be such that $E^{\bot}\subset H$. Let
$\phi:\R^n\longrightarrow[0,\infty)$ be a $(-1/n)$-concave function and
let $\mu$ be the measure on $\R^n$ given by $\dlat\mu(x)=\phi(x)\,\dlat
x$. Then, for every  $K\in\K^n$,
\begin{equation}\label{e:RS_E_H}
\sup_{x\in E^{\bot}}\mu_i\bigl(K\cap (x+E)\bigr)\sup_{y\in
H^{\bot}}\mu_j\bigl(K\cap (y+H)\bigr)
\leq\binom{n-k}{n-i}\sup_{x\in\R^n}\mu_k\bigl(K\cap (x+F)\bigr)\mu(K),
\end{equation}
where $F=E\cap H$.
\end{corollary}

\begin{proof}
Let $f:F^{\bot}\longrightarrow[0,\infty)$ be the function given by
\[
f(x,y)=\int_{\R^k}\phi(x,y,z)\chi_{_{K}}(x,y,z) \,\dlat z.
\]
The Borell-Brascamp-Lieb inequality (see e.g. \cite[Theorem 10.1]{G})
implies that $f$ is quasi-concave and, in particular, $\c_t(f)$ is a convex
body. Then, we may apply Corollary \ref{c:RS_sect_proj} to obtain
\begin{equation*}
\begin{split}
& \int_{E^{\bot}}\sup_{y\in H^{\bot}}\int_{\R^k}
    \phi(x,y,z) \chi_{_{K}}(x,y,z) \,\dlat z\,\dlat x
    \sup_{x\in E^{\bot}}\int_{H^{\bot}}\int_{\R^k}\phi(x,y,z)\chi_{_{K}}(x,y,z) \,\dlat z\,\dlat y\\
 & \leq \binom{n-k}{n-i}\sup_{(x,y)\in F^{\bot}}\int_{\R^k}\phi(x,y,z)\chi_{_{K}}(x,y,z) \,\dlat z
    \int_{F^\bot}\int_{\R^k}\phi(x,y,z)\chi_{_{K}}(x,y,z) \,\dlat z\,\dlat x\,\dlat y
\end{split}
\end{equation*}
and thus, in particular, for every $y_0\in H^{\bot}$ we have
\begin{equation*}
\begin{split}
\int_{E^{\bot}}\int_{\R^k}\phi(x,y_0,z)\chi_{_{K}}(x,y_0,&z)
    \,\dlat z\,\dlat x
    \sup_{x\in E^{\bot}}\int_{H^{\bot}}\int_{\R^k}\phi(x,y,z)\chi_{_{K}}(x,y,z) \,\dlat z\,\dlat y\\
 & \leq \binom{n-k}{n-i}\sup_{(x,y)\in F^{\bot}}\mu_k\Bigl(K\cap\bigl((x,y)+F\bigr)\Bigr)\mu(K).
\end{split}
\end{equation*}
Hence, for every $y_0\in H^{\bot}$, we get
\begin{equation*}
\mu_j\bigl(K\cap (y_0+ H)\bigr)\sup_{x\in E^{\bot}}\mu_i\bigl(K\cap
(x+E)\bigr)\leq \binom{n-k}{n-i}\sup_{x\in\R^n}\mu_k\bigl(K\cap
(x+F)\bigr)\mu(K),
\end{equation*}
which  implies \eqref{e:RS_E_H}.
\end{proof}

Next we show how one may exploit the approach we are following in this
section to obtain an analogous result to Proposition
\ref{t:RS_omega_rad_decreasing}, in the setting of quasi-concave densities
which are not necessarily continuous. Notice that whereas the right-hand
side in \eqref{e:RS_omega_quasiconcave} is smaller than the right-hand
side in \eqref{e:RS_omega_rad_decreasing}, the constants $c(\omega)$ and
$\phi(\omega)/\|\phi\|_{\infty}$ are not comparable in general.

\begin{theorem}\label{t:RS_omega_quasiconcave}
Let $K\in\K^n$ and let $\mu$ be a measure on $\R^n$ given by
$\dlat\mu(x)=\phi(x)\,\dlat x$, where $\phi:\R^n\longrightarrow[0,\infty)$
is a bounded quasi-concave function. Then, for every $\omega\in\R^n$,
\begin{equation}\label{e:RS_omega_quasiconcave}
\frac{\phi(\omega)}{\|\phi\|_{\infty}}\mu(K-K+\omega)
\leq\binom{2n}{n}\min\left\{\sup_{y\in K}\mu(y+\omega-K),
    \sup_{y\in K}\mu(-y+\omega+K)\right\}.
\end{equation}
Moreover, if $\phi$ is continuous at $\omega_0$, for some
$\omega_0\in\R^n$, then equality holds in \eqref{e:RS_omega_quasiconcave}
(for such $\omega_0$) if and only if $\mu$ is a constant multiple of the
Lebesgue measure on $K-K+\omega_0$, $\phi(\omega_0)=\|\phi\|_{\infty}$ and
$K$ is a simplex.
\end{theorem}
\begin{proof}
Let $\omega\in\R^n$ and consider the function
$f_\omega:K-K+\omega\longrightarrow[0,\infty)$ given by
\[
f_\omega(x)=\vol\bigl(K\cap(x-\omega+K)\bigr).
\]
Notice that, $f_\omega$ is $(1/n)$-concave by \eqref{e:B-M_ineq}, $\supp
f_\omega=K-K+\omega$ and, moreover, that
$\|f_\omega\|_{\infty}=f_\omega(\omega)=\vol(K)$. Then, using
\eqref{e:mu(suppf)_x0}, we get
\begin{equation*}
\begin{split}
\frac{\phi(\omega)}{\|\phi\|_{\infty}}\mu(K-K+\omega) &
    \leq\binom{2n}{n}\frac{1}{\vol(K)}\int_{\R^n}\vol\bigl(K\cap(x-\omega+K)\bigr)\,\dlat \mu(x)\\
 & =\binom{2n}{n}\frac{1}{\vol(K)}\int_{\R^n}\phi(x)\int_{\R^n}\chi_{_K}(y)\chi_{_{y+\omega-K}}(x)\,\dlat y\,\dlat x\\
 & =\binom{2n}{n}\frac{1}{\vol(K)}\int_{K}\mu(y+\omega-K)\,\dlat y\leq\binom{2n}{n}\sup_{y\in K}\mu(y+\omega-K).
\end{split}
\end{equation*}
Therefore, exchanging the roles of $K$ and $-K$,
\eqref{e:RS_omega_quasiconcave} infers.

Finally, if equality holds in \eqref{e:RS_omega_quasiconcave} for some
$\omega_0\in\R^n$ then, by Corollary~\ref{c:|f|_x0}, $\mu$ is a constant
multiple of the Lebesgue measure on $K-K+\omega_0$ and
$\phi(\omega_0)=\|\phi\|_{\infty}$. Now, from the equality case of Theorem
\ref{t:RS}, $K$ must be a simplex. The converse is immediate from Theorem
\ref{t:RS}.
\end{proof}

We conclude this section by noticing that, from the proof of the previous
result, one may also obtain \eqref{e:RS_measures_rad_decreasing} in the
slightly less general setting of quasi-concave densities with maximum at
the origin. We include it here for the sake of completeness.
\begin{corollary}\label{c:RS_measures}
Let $K\in\K^n$ and let $\mu$ be the measure on $\R^n$ given by
$\dlat\mu(x)=\phi(x)\,\dlat x$, where $\phi:\R^n\longrightarrow[0,\infty)$
is a quasi-concave function with $\|\phi\|_{\infty}=\phi(0)$. Then
\begin{equation*}\label{e:RS_measures}
\mu(K-K)\leq
\binom{2n}{n}\min\bigl\{\overline{\mu}(K),\overline{\mu}(-K)\bigr\}.
\end{equation*}
Moreover, if $\phi$ is continuous at the origin then equality holds if and
only if $\mu$ is a constant multiple of the Lebesgue measure on $K-K$ and
$K$ is a simplex.
\end{corollary}

\section{A remark for measures with $p$-concave densities, $p>0$}\label{s:remark}

As we have shown in Example \ref{r:hip_P_HK}, the assumption $P_HK\subset
K$ on Theorems~\ref{t:RS_secc_proy_K(0)} and \ref{t:RS_seccion_proy_quasi}
is necessary. However, when dealing with measures associated to
$p$-concave densities, $p>0$, an inequality in the spirit of
\eqref{e:RS_section_proy} can be obtained for an arbitrary $K\in\K^n$, by
setting a binomial coefficient according to the concavity nature of the
density. This is the content of the following result.

\begin{theorem}\label{t:RS_seccion_proy_p_conc}
Let $k\in\{1,\dots,n-1\}$, $r\in\N$ and $H\in\G(n,n-k)$. Given a
$(1/r)$-concave function $\phi_k:\R^k\longrightarrow[0,\infty)$, and a
radially decreasing function
$\phi_{n-k}:\R^{n-k}\longrightarrow[0,\infty)$, let
$\mu_n=\mu_{n-k}\times\mu_{k}$ be the product measure on $\R^n$ given by
$\dlat\mu_{n-k}(x)=\phi_{n-k}(x)\,\dlat x$ and
$\dlat\mu_{k}(y)=\phi_k(y)\,\dlat y$. Let $K\in\K^n$ be such that
$\max_{x\in H}
\mu_k\left(K\cap\bigl(x+H^{\bot}\bigr)\right)=\mu_k\left(K\cap
H^{\bot}\right)$. Then
\begin{equation*}
\mu_{n-k}\bigl(P_HK\bigr)\mu_k\bigl(K\cap
H^{\bot}\bigr)\leq\binom{n+r}{n-k}\mu_n(K).
\end{equation*}
\end{theorem}

\begin{proof}
Consider the function $f:H\longrightarrow\R$ given by
\[
f(x)=\mu_k\left(K\cap\bigl(x+H^{\bot}\bigr)\right),
\]
which satisfies $\supp f=P_HK$.

Now, the Borell-Brascamp-Lieb inequality (see \cite[Theorem 10.1]{G})
implies that $\mu_k$ is ($1/(k+r)$)-concave which, together with the
convexity of $K$, yields that $f$ is ($1/(k+r)$)-concave. Furthermore, by
assumption we have that $\|f\|_{\infty}=f(0)$. Thus, using
\eqref{e:int(suppf_C_theta)_g_rad_dec} for $g=\phi_{n-k}$, we obtain
\[
\alpha^{n-k}_{1/(k+r),0}\int_{P_HK}\phi_{n-k}(x)\,\dlat
x\leq\dfrac{1}{\mu_k\bigl(K\cap H^{\bot}\bigr)} \int_{P_HK}
\mu_k\left(K\cap\bigl(x+H^{\bot}\bigr)\right)\,\phi_{n-k}(x)\dlat x
\]
and hence
\begin{equation*}
\mu_{n-k}\bigl(P_HK\bigr)\mu_k\bigl(K\cap
H^{\bot}\bigr)\leq\binom{n+r}{n-k}\mu_n(K),
\end{equation*}
as desired.
\end{proof}

The latter result can be stated for any positive real number $r$, just
replacing $\binom{n+r}{n-k}$ by the suitable constant.

We notice that the above inequality includes \eqref{e:RS_section_proy} as
a special case, since the constant density (of the Lebesgue measure) is
$\infty$-concave, and thus $r=0$.

\medskip

\noindent {\it Acknowledgements.} We thank the referees for many valuable
suggestions and remarks which have allowed us to considerably improve the
manuscript.

%


\begin{thebibliography}{99}

\bibitem{AAGJV} Alonso-Guti\'errez, D., Artstein-Avidan, S., Gonz\'alez,
B., Jim\'enez, C. H. and Villa, R. ``Rogers-Shephard and local
Loomis-Whitney type inequalities.'' {\it Submitted},
\href{https://arxiv.org/abs/1706.01499v2}{arXiv:1706.01499v2}.

\bibitem{AlGMJV} Alonso-Guti\'errez, D., Gonz\'alez, B., Jim\'enez,
C. H. and Villa, R. ``Rogers-Shephard inequality for log-concave
functions.'' {\it J. Func. Anal.} 271, no. 11 (2016): 3269--3299.

\bibitem{AGM} Artstein-Avidan, S., Giannopoulos, A. and  Milman, V. D.
{\it Asymptotic geometric analysis. Part I}. Mathematical Surveys and
Monographs, 202. Providence, RI: American Mathematical Society, 2015.

\bibitem{AKM} Artstein-Avidan, S., Klartag, B. and Milman, V. ``The
Santal\'o point of a function, and a functional form of the Santal\'o
inequality.'' {\it Mathematika} 51 (2004): 33--48.

\bibitem{Ba1} Ball, K. {\it Isometric problems in $\ell_p$ and sections of
convex sets}. PhD dissertation. Cambridge: 1986.

\bibitem{Ba2} Ball, K. ``Logarithmically concave functions and sections of
convex sets in $\R^n$.'' {\it Studia Math.} 88, no. 1 (1988): 69--84.

\bibitem{Borell1} Borell, C. ``Convex measures on locally convex spaces.''
{\it Ark. Mat.} 12 (1974): 239--252.

\bibitem{Borell} Borell, C. ``Convex set functions in $d$-space.'' {\it Period.
Math. Hungar.} 6 (1975): 111--136.

\bibitem{B2} Borell, C. ``The Brunn-Minkowski inequality in Gauss space.''
{\it Invent. Math.}  30, no. 2 (1975): 207--216.

\bibitem{B3} Borell, C. ``The Ehrhard inequality.''  {\it C. R. Math. Acad.
Sci. Paris} 337, no. 10 (2003): 663--666.

\bibitem{BL} Brascamp, H. J. and Lieb, E. H. ``On extensions of the
Brunn-Minkowski and Pr\'ekopa-Leindler theorems, including inequalities
for log concave functions and with an application to the diffusion
equation.'' {\it  J. Func. Anal.} 22, no. 4 (1976): 366--389.

\bibitem{Co} Colesanti, A. ``Functional inequalities related to the
Rogers-Shephard inequality.'' {\it Mathematika} 53 (2006): 81--101.

\bibitem{E1} Ehrhard, E. ``Sym\'etrisation dans l'espace de Gauss.'' {\it
Math. Scand.} 53 (1983): 281--301.

\bibitem{E2} Ehrhard, E. ``\'Elements extr\'emaux pours les in\'egalit\'es
de Brunn-Minkowski gaussienes.'' {\it Ann. Inst. H. Poincar\'e Probab.
Statist.} 22 (1986): 149--168.

\bibitem{FM} Fradelizi, M. and Meyer, M. ``Some functional forms of
Blaschke-Santal\'o inequality.'' {\it Math. Z.} 256 (2007):
379--395.

\bibitem{G} Gardner, R. J. ``The Brunn-Minkowski inequality.'' {\it Bull. Amer.
Math. Soc.} 39, no. 3 (2002): 355--405.

\bibitem{Ga} Gardner, R. J. {\it  Geometric tomography}, 2nd ed.
Encyclopedia of Mathematics and its Applications, 58. Cambridge: Cambridge
University Press, 2006.

\bibitem{GaZv} Gardner, R. J. and Zvavitch, A. ``Gaussian Brunn-Minkowski
inequalities.'' {\it Trans. Amer. Math. Soc.} 362, no. 10 (2010):
5333--5353.

\bibitem{KK} Klartag, B. and Koldobsky, K. ``An example related to the
slicing inequality for general measures.'' {\it J. Funct. Analysis} 274,
no. 7 (2018): 2089--2112.

\bibitem{KLi} Klartag, B. and Livshyts, G. ``The lower bound for Koldobsky's
slicing inequality via random rounding.'' {\it Submitted},
\href{https://arxiv.org/abs/1810.06189}{arXiv:1810.06189}.

\bibitem{KM} Klartag, B. and Milman, V. ``Geometry of log-concave functions
and measures.'' {\it Geom. Dedicata} 112 (2005): 169--182.

\bibitem{Kol} Koldobsky, A. ``Slicing inequalities for measures of convex bodies.''
{\it Adv. Math.} 283 (2015): 473--488.

\bibitem{KoZ} Koldobsky, A. and  Zvavitch, A. ``An isomorphic version of
the Busemann-Petty problem for arbitrary measures.'' {\it Geom. Dedicata}
174 (2015): 261--277.

\bibitem{Leindler} Leindler, L. ``On certain converse of H\"older's
inequality II.'' {\it Acta Sci. Math. (Szeged)} 33 (1972): 217--223.

\bibitem{Liv} Livshyts, G. ``An extension of Minkowski's theorem and its
applications to questions about projections for measures.'' {\it To appear
in Adv. Math.}

\bibitem{LiMaNaZv} Livshyts, G., Marsiglietti, A., Nayar, P. and Zvavitch,
A. ``On the Brunn-Minkowski inequality for general measures with
applications to new isoperimetric-type inequalities.'' {\it Trans. Amer.
Math. Soc.} 369, no. 12 (2017): 8725--8742.

\bibitem{Mar} Marsiglietti, A. ``On the improvement of concavity of convex
measures.'' {\it Proc. Amer. Math. Soc.} 144, no. 2 (2016): 775--786.

\bibitem{MNRY} Meyer, M., Nazarov, F., Ryabogin, D. and Yaskin, V.
``Gr\"unbaum-type inequality for log-concave functions.'' {\it Bull. Lond.
Math. Soc.} 50 (2018): 745--752.

\bibitem{MSZ} Myroshnychenko, S., Stephen, M. and Zhang, N.
``Gr\"unbaum's inequality for sections.'' {\it J. Funct. Anal.} 275
(2018): 2516--2537.

\bibitem{NaTk} Nayar, P. and Tkocz, T. ``A note on a Brunn-Minkowski inequality
for the Gaussian measure.'' {\it Proc. Amer. Math. Soc.} 141 (2013):
4027--4030.

\bibitem{Prekopa} Pr\'ekopa, A. ``Logarithmic concave measures with
application to stochastic programming.'' {\it Acta Sci. Math. (Szeged)} 32
(1971): 301--315.

\bibitem{RYN} Ritor\'e, M. and Yepes Nicol\'as, J. ``Brunn-Minkowski inequalities
in product metric measure spaces.'' {\it Adv. Math.} 325 (2018): 824--863.

\bibitem{RoWe} Rockafellar, R. T. and Wets, R. J.-B. {\it Variational
analysis}. Grundlehren der Mathematischen Wissenschaften [Fundamental
Principles of Mathematical Sciences], 317. Berlin: Springer-Verlag, 1998.

\bibitem{RS1} Rogers, C. A. and Shephard, G. C. ``The difference body of a convex
body.'' {\it Arch. Math.} 8 (1957): 220--233.

\bibitem{RS2} Rogers, C. A. and Shephard, G. C. ``Convex bodies associated with
a given convex body.'' {\it J. Lond. Math. Soc.} 1, no. 3 (1958):
270--281.

\bibitem{Sch} Schneider, R. {\it Convex bodies: The Brunn-Minkowski
theory}, 2nd expanded ed. Encyclopedia of Mathematics and its
Applications, 151. Cambridge: Cambridge University Press, 2014.

\bibitem{ST} Sudakov, V. N. and Cirel'son, B. S. ``Extremal properties of
half-spaces for spherically invariant measures.'' Problems in the theory
of probability distributions, II. {\it Zap. Nau\v{c}n. Sem. Leningrad.
Otdel. Mat. Inst. Steklov. (LOMI)} 41 (1974): 14--24, 165.

\bibitem{Z} Zvavitch, A. ``The Busemann-Petty problem for arbitrary
measures.'' {\it Math. Ann.} 331, no. 4 (2005): 867--887.

\end{thebibliography}
\end{document}